\documentclass[11pt]{article}
\usepackage[colorlinks=true,linkcolor=blue,citecolor=Red]{hyperref}
\usepackage{breakcites}

\usepackage{mathtools}

\usepackage{subfigure}
\usepackage{stackengine}

\usepackage{algorithm}
\usepackage{algpseudocode}

\usepackage{latexsym}
\usepackage{graphics}
\usepackage{amsmath}
\usepackage[algo2e]{algorithm2e}
\usepackage{xspace}
\usepackage{amssymb}
\usepackage{psfrag}
\usepackage{epsfig}
\usepackage{amsthm}
\usepackage{url}
\usepackage{pst-all}
\usepackage{bbm}

\usepackage[title]{appendix}

\usepackage{color}

\definecolor{Red}{rgb}{1,0,0}
\definecolor{Blue}{rgb}{0,0,1}
\definecolor{Olive}{rgb}{0.41,0.55,0.13}
\definecolor{Yarok}{rgb}{0,0.5,0}
\definecolor{Green}{rgb}{0,1,0}
\definecolor{MGreen}{rgb}{0,0.8,0}
\definecolor{DGreen}{rgb}{0,0.55,0}
\definecolor{Yellow}{rgb}{1,1,0}
\definecolor{Cyan}{rgb}{0,1,1}
\definecolor{Magenta}{rgb}{1,0,1}
\definecolor{Orange}{rgb}{1,.5,0}
\definecolor{Violet}{rgb}{.5,0,.5}
\definecolor{Purple}{rgb}{.75,0,.25}
\definecolor{Brown}{rgb}{.75,.5,.25}
\definecolor{Grey}{rgb}{.5,.5,.5}

\newcommand{\bs}{\boldsymbol{\sigma}}

\usepackage{bbm}
\newcommand{\ind}{\mathbbm{1}}

\def\OPT{{\mathsf{H^*}}}

\newcommand{\R}{\mathbb{R}}

\newcommand{\N}{\mathbb{N}}
\newcommand{\ip}[2]{\langle{#1},{#2}\rangle} 
\renewcommand{\ip}[2]{\left\langle#1,#2\right\rangle}

\renewcommand{\R}{\mathbb{R}}

\newcommand{\cN}{{\bf \mathcal{N}}}

\newcommand{\ignore}[1]{\relax}

\newtheorem{theorem}{Theorem}[section]
\newtheorem{remark}[theorem]{Remark}
\newtheorem{lemma}[theorem]{Lemma}

\newtheorem{proposition}[theorem]{Proposition}

\newtheorem{definition}[theorem]{Definition}

\renewcommand{\ip}[2]{\left\langle#1,#2\right\rangle}


\newcounter{parentnumber}

\makeatletter
\def\BState{\State\hskip-\ALG@thistlm}
\makeatother

\definecolor{Red}{rgb}{1,0,0}
\definecolor{Blue}{rgb}{0,0,1}
\definecolor{Olive}{rgb}{0.41,0.55,0.13}
\definecolor{Green}{rgb}{0,1,0}
\definecolor{MGreen}{rgb}{0,0.8,0}
\definecolor{DGreen}{rgb}{0,0.55,0}
\definecolor{Yellow}{rgb}{1,1,0}
\definecolor{Cyan}{rgb}{0,1,1}
\definecolor{Magenta}{rgb}{1,0,1}
\definecolor{Orange}{rgb}{1,.5,0}
\definecolor{Violet}{rgb}{.5,0,.5}
\definecolor{Purple}{rgb}{.75,0,.25}
\definecolor{Brown}{rgb}{.75,.5,.25}
\definecolor{Grey}{rgb}{.5,.5,.5}
\definecolor{Pink}{rgb}{1,0,1}
\definecolor{DBrown}{rgb}{.5,.34,.16}
\definecolor{Black}{rgb}{0,0,0}

\usepackage{float}

\let\OLDthebibliography\thebibliography
\renewcommand\thebibliography[1]{
  \OLDthebibliography{#1}
  \setlength{\parskip}{0pt}
  \setlength{\itemsep}{0pt plus 0.3ex}
}

\usepackage{float}
\title{Sharp Thresholds for the Overlap Gap Property: Ising 
 $p$-Spin Glass and Random $k$-SAT}
\author{{\sf Eren C. K{\i}z{\i}lda\u{g}}\thanks{Department of Statistics, University of Illinois at Urbana-Champaign; e-mail: {\tt kizildag@illinois.edu}}}
\begin{document}
\maketitle
\begin{abstract}
The Ising $p$-spin glass and random $k$-SAT are two canonical examples of disordered systems that play a central role in understanding the link between geometric features of optimization landscapes and computational tractability. Both models exhibit hard regimes where all known polynomial-time algorithms fail and possess the multi Overlap Gap Property ($m$-OGP), an intricate geometrical property that rigorously rules out a broad class of algorithms exhibiting input stability. 

We establish that, in both models, the symmetric $m$-OGP undergoes a sharp phase transition, and we pinpoint its exact threshold. For the Ising $p$-spin glass, our results hold for all sufficiently large $p$; for the random $k$-SAT, they apply to all $k$ growing mildly with the number of Boolean variables. Notably, our findings yield qualitative insights into the power of OGP-based arguments. A particular consequence for the Ising $p$-spin glass is that the strength of the $m$-OGP in establishing algorithmic hardness grows without bound as $m$ increases. 

These are the first sharp threshold results for the $m$-OGP. Our analysis hinges on a judicious application of the second moment method, enhanced by concentration. While a direct second moment calculation fails, we overcome this via a refined approach that leverages an argument of~\cite{frieze1990independence} and exploiting concentration properties of carefully constructed random variables.
\end{abstract}
\newpage
\tableofcontents
\newpage
%\pagenumbering{arabic}
%\section{Introduction}
\section{Introduction}\label{sec:Intro}
In this paper, we focus on two fundamental models in the statistical mechanics of disordered systems: Ising $p$-spin glass and random $k$-SAT. Both models have been extensively studied across probability, theoretical computer science, and physics; see~\cite{mezard2009information,stein2013spin} and references therein. Sharing profound structural similarities, including rich phase transitions and rugged landscapes, they are central to understanding the interplay between the geometrical properties of optimization landscapes and computational phase transitions in average-case settings. Our particular focus is on their inherent algorithmic barriers.
%Our particular focus is on the algorithmic barriers inherent in these models. 

Both models exhibit regimes where no polynomial-time algorithms are known to succeed. For the Ising $p$-spin glass, which involves optimizing random polynomials, existing algorithms find strictly suboptimal solutions. For the random $k$-SAT, a random constraint satisfaction problem (CSP), known algorithms operate only at constraint densities far below the satisfiability threshold. Due to their average-case nature, classical complexity theory offers little insight into these barriers. An extensive body of work has instead sought unconditional lower bounds, showing that both models exhibit the multi Overlap Gap Property ($m$-OGP)---an intricate geometrical feature of the optimization landscape that rigorously rules out a broad class of algorithms~\cite{gamarnik2017performance,chen2019suboptimality,gamarnik2020low,bresler2021algorithmic,huang2021tight,gamarnik2023shattering,gamarnik2021overlap}. These works have focused almost exclusively on 
establishing such lower bounds via the OGP and exploring their tightness---i.e., whether they match known algorithmic guarantees---often by analyzing increasingly refined versions of the OGP as $m$ grows. For fixed $m\in\mathbb{N}$, a canonical result asserts that there exists a threshold $\gamma(m)$ such that when $\gamma>\gamma(m)$, the corresponding model exhibits $m$-OGP and certain algorithms necessarily fail; here, $\gamma$ denotes a natural model parameter. However, a finer, two-sided understanding of the link between OGP and hardness remains elusive. In some models, OGP-based lower bounds match provable average-case hardness~\cite{vafa2025symmetric}. At the same time, certain optimization problems exhibit the OGP, yet still admit efficient algorithms~\cite{li2024some}. These examples collectively highlight the need for a more nuanced understanding of OGP-based methods---one that addresses not only how OGP rules out algorithms, but also how its power evolves with $m$ and what its absence implies.
%These examples collectively demonstrate that toward a complete average-case theory, a more nuanced understanding of OGP-based methods is necessary. This includes not only leveraging OGP to rule out algorithms, but also understanding how the power of OGP-based methods evolves with $m$ and the implications of the absence of OGP.

To this end, a natural question arises: does the OGP indeed become more powerful as $m$ increases? This is directly tied to understanding the regime where $m$-OGP is absent, as elaborated below. Typically, the presence of OGP is established via the first moment method, with the threshold $\gamma(m)$ often improving with $m$. Consequently, a common rule of thumb in the field is that increasing $m$ leads to stronger algorithmic lower bounds.\footnote{A similar monotonicity in $m$ has also been conjectured in the context of spherical spin glasses~\cite{arous2021shattering}.} However, this approach offers no insight into the regime $\gamma < \gamma(m)$. To \emph{truly} establish that the power of $m$-OGP grows with $m$, one must precisely locate its threshold---by showing that $m$-OGP is present for $\gamma>\gamma(m)$ and absent for $\gamma<\gamma(m)$. Without this, the observed improvement in $\gamma(m)$ 
could be spurious, merely an artifact of the first moment method. Such a sharp analysis would also rigorously justify the prevailing practice of increasing $m$ to tighten lower bounds. This issue is particularly stark in the binary perceptron, where the first moment method fails: in certain regimes, the predicted $m$-OGP threshold actually worsens with $m$, and the very presence of OGP remains unknown for $m\ge 4$~\cite[Section~II.B]{gamarnik2022algorithms}. Moreover, the absence of OGP is not merely of theoretical interest: for the Sherrington-Kirkpatrick (SK) model, algorithms are known to succeed under the widely believed assumption that OGP is absent~\cite{montanari2019FOCS}.  These examples underscore that understanding both the presence and absence of OGP—especially in the regime $\gamma < \gamma(m)$—is essential to fully grasp its power and limitations.
 %, and serve as our main motivation.
%the first moment method does fail for a key model known as binary perceptron: the $m$-OGP threshold $\gamma(m)$ as prescribed by the first moment method actually worsens for $m\ge 4$ in a certain regime~\cite[Section~II.B]{gamarnik2022algorithms}. 
%As such, 
%These points underscore the necessity of  to fundamentally . This challenge serves as our main motivation; we address it in our work.

In this paper, we pinpoint the precise threshold for the symmetric $m$-OGP---a widely used variant of the OGP (see below)---for both models, thereby demonstrating a sharp phase transition. These thresholds are particularly transparent and strictly monotonic in $m$, in the large-$p$ and large-$k$ regimes. Our results yield qualitative insights into the power of OGP-based arguments. Notably, for the $p$-spin model, we establish that the power of $m$-OGP in proving algorithmic hardness increases indefinitely with $m$. In the context of spin glasses and shattering, our results also explain why a certain \emph{soft-OGP} argument by El Alaoui~\cite{alaoui2024near} is necessary; see Section~\ref{sec:Overview}.

For an overview of our results, see Section~\ref{sec:Overview}. We next  formally introduce the two models.
%For an overview of our results, also see Section~\ref{sec:Overview}. We now formally introduce the two models.
%See  for an overview of our results. 
%For a fixed $m\in\mathbb{N}$, a canonical result, often established via the first moment method, states that there exists a threshold $\gamma(m)$ such that when $\gamma>\gamma(m)$, the corresponding model exhibits $m$-OGP and certain classes of algorithms fail, where $\gamma$ is a natural parameter of the model. However, the regime $\gamma<\gamma(m)$ remains largely unexplored. A fundamental question still remains: can we precisely determine the threshold for the onset of the $m$-OGP? Is the $m$-OGP absent when $\gamma<\gamma(m)$? Our main contribution is to calculate the precise location of the $m$-OGP threshold for both models, demonstrating that the $m$-OGP undergoes a sharp phase transition. These thresholds are easier to interpret in the regime of large $p$ and $k$. Our results yield qualitative insights into the power and the behavior of OGP-based arguments. For instance, for the $p$-spin model, our results imply that in the limit of large $p$, the power of $m$-OGP in establishing algorithmic hardness improves indefinitely with $m$. See Section~\ref{sec:Overview} for an overview. 
% we investigate.
%We focus on the Ising pure $p$-spin glass and the random $k$-SAT models. Both of these models known to exhibit the multi Overlap Gap Property ($m$-OGP), an intricate geometrical property which is a rigorous barrier against important classes of algorithms (see below). 
\subsection{Ising Pure $p$-Spin Glass Model}
For a fixed integer $p\ge 2$, an order-$p$ tensor $\boldsymbol{J}=(J_{i_1,\dots,i_p}:1\le i_1,\dots,i_p\le n)$ with i.i.d.\,standard normal entries, $J_{i_1,\dots,i_p}\sim \cN(0,1)$, and $\bs\in\R^n$, the $p$-spin Hamiltonian is given by
\begin{equation}\label{eq:Hamiltonian-p-spin}
    H_{n,p}(\bs)=n^{-\frac{p+1}{2}}\ip{\boldsymbol{J}}{\boldsymbol{\bs}^{\otimes p}} = n^{-\frac{p+1}{2}}\sum_{1\le i_1,\dots,i_p\le n}J_{i_1,\dots,i_p}\bs_{i_1}\cdots\bs_{i_p}.
\end{equation}
Introduced by Derrida~\cite{derrida1980random}, the $p$-spin model generalizes the SK model ($p=2$)~\cite{sherrington1975solvable} by incorporating $p$-body interactions. 
%The $p$-spin model was introduced by Derrida~\cite{derrida1980random}. By incorporating $p$-body interactions, it generalizes the Sherrington-Kirkpatrick (SK) model~\cite{sherrington1975solvable}, the case $p=2$ in~\eqref{eq:Hamiltonian-p-spin}. 
Our particular focus is on the \emph{Ising} case, where the configuration space is the discrete hypercube, $\bs\in \Sigma_n:=\{-1,1\}^n$.\footnote{The \emph{spherical} case where $\bs$ lies in the hypersphere of radius $\sqrt{n}$, $\|\bs\|_2=\sqrt{n}$, is also studied extensively.} A key quantity regarding spin glass models is the limiting \emph{free energy} $F(\beta):= \lim_{n\to\infty}\ln Z_\beta/n$, where $\quad Z_\beta = \sum_{\bs\in\Sigma_n}\exp(\beta n H_{n,p}(\bs))$ at \emph{inverse temperature}
%\begin{equation}\label{eq:Free-energy}
 %   F(\beta):= \lim_{n\to\infty}\frac{\ln Z_\beta}{n},\quad\text{where}\quad Z_\beta = \sum_{\bs\in\Sigma_n}e^{\beta n H_{n,p}(\bs)},
%\end{equation}
%where 
$\beta>0$. The limiting free energy for the SK model was predicted in a celebrated work of Parisi~\cite{parisi1979infinite} using the non-rigorous replica method. Following the work of Guerra~\cite{guerra2003broken}, Talagrand~\cite{talagrand2006parisi} rigorously verified Parisi's prediction in a breakthrough paper; his proof in fact extends to all even $p$. For general $p$, Parisi formula was verified by Panchenko~\cite{panchenko2014parisi}, following the work Aizenman, Sims, and Starr~\cite{aizenman2003extended}. The existence of the free energy limit was shown earlier by Guerra and Toninelli~\cite{guerra2002thermodynamic} using the so-called interpolation method. Far more is known rigorously regarding the $p$-spin glass models, including the ultrametricity of the asymptotic Gibbs measure~\cite{panchenko2013parisi,jagannath2017approximate,chatterjee2021average}, the TAP equations~\cite{auffinger2019thouless} and the TAP free energy~\cite{subag2018free,chen2023generalized}, a certain phase diagram~\cite{toninelli2002almeida,auffinger2015properties,jagannath2017some}, shattering~\cite{arous2021shattering,alaoui2023shattering,gamarnik2023shattering}, and algorithmic results (see below). Furthermore, more general spin glass models have also been studied in the literature. A particularly relevant case is the \emph{vector spin models} where the states are tuples of points in $\R^n$ (as opposed to points in $\R^n$) with a prescribed overlap, see~\cite{panchenko2018free,ko2020free}.
For further pointers to relevant literature, see the textbooks~\cite{talagrand2010mean,talagrand2011mean,panchenko2013sherrington}.
\paragraph{Finding a Near-Ground State in Polynomial-Time} Define the \emph{ground-state energy} by $\OPT:= \lim_{n\to\infty}\max_{\bs\in\Sigma_n} H_{n,p}(\bs)$,
%\begin{equation*}%\label{eq:Ground-En}
   % \OPT:= \lim_{n\to\infty}\max_{\bs\in\Sigma_n} H_{n,p}(\bs),
%\end{equation*}
which can be recovered as the zero temperature limit of the Parisi formula, $\OPT=\lim_{\beta\to\infty}F(\beta)/\beta$.\footnote{See also~\cite{auffingerchen} for a generalized
Parisi formula for $\OPT$ that does not require $\beta\to\infty$ limit.} Given $\OPT$, an algorithmic question is to find a near ground-state efficiently. That is, given $\boldsymbol{J}\in(\R^n)^{\otimes p}$ and $\epsilon>0$, find in polynomial time a $\bs_{\mathrm{ALG}}\in\Sigma_n$ such that $H_{n,p}(\bs_{\mathrm{ALG}})\ge (1-\epsilon)\OPT$, ideally with high probability (w.h.p.). For the SK model, this problem was solved by Montanari~\cite{montanari2019FOCS}, who developed an Approximate Message Passing (AMP) type algorithm. Montanari's algorithm was inspired by an algorithm of Subag~\cite{subag2021following} for spherical mixed $p$-spin model; it is contingent on a widely believed conjecture that the SK model does not exhibit the OGP. More recently,~\cite{el2021optimization,sellke2021optimizing} extended Montanari's algorithm to Ising mixed $p$-spin models. For any $\epsilon>0$, these algorithms also return a $\bs_{\mathrm{ALG}}$ for which $H_{n,p}(\bs_{\mathrm{ALG}})\ge (1-\epsilon)\OPT$ w.h.p., provided the underlying mixed $p$-spin model does not exhibit the OGP. It is known, however, that the Ising $p$-spin glass model exhibits the OGP for $p\ge 4$~\cite[Theorem~3]{chen2019suboptimality}, and that the OGP serves a rigorous barrier to AMP type algorithms~\cite{gamarnikjagannath2021overlap}. Specifically, for any even $p\ge 4$, there exists a value $\bar{\mu}$ such that no AMP type algorithm can find a $\bs_{\mathrm{ALG}}$ with $H_{n,p}(\bs_{\mathrm{ALG}})\ge \OPT-\bar{\mu}$~\cite{gamarnikjagannath2021overlap}. This algorithmic lower bound was later extended to low-degree polynomials~\cite{gamarnik2020low} and low-depth Boolean circuits~\cite{gamarnik2021circuit}. However, classical OGP~\cite{chen2019suboptimality} and the aforementioned lower bounds do not match the best known algorithmic threshold. More recently, Huang and Sellke~\cite{huang2021tight,huang2023algorithmic} introduced a very elegant \emph{branching OGP} consisting of an ultrametric tree of solutions and subsequently obtained tight bounds for Lipschitz algorithms.
\subsection{Random $k$-SAT Model} 
Given Boolean variables $x_1,\dots,x_n$, a $k$-clause $\mathcal{C}$ is defined as a disjunction of $k$ literals $\{y_1,\dots,y_k\}$ chosen from $\{x_1,\dots,x_n,\bar{x}_1,\dots,\bar{x}_n\}$, i.e., $\mathcal{C}=y_1 \vee \cdots \vee y_k$. A random $k$-SAT formula $\Phi$ is a conjunction of $M$ $k$-clauses $\mathcal{C}_1,\dots,\mathcal{C}_M$, with each of its $kM$ literals sampled independently and uniformly from $\{x_1,\dots,x_n,\bar{x}_1,\dots,\bar{x}_n\}$: $\Phi = \mathcal{C}_1\wedge \cdots \wedge \mathcal{C}_M$.\footnote{This model allows clauses to contain repeated literals, or both a variable $x_i$ and its negation $\overline{x}_i$. However, we focus on the regime $k=\Omega(\ln n)$, where such occurrences are rare as $n\to\infty$. This is standard in the literature~\cite{frieze2005random,coja2008random,bresler2021algorithmic}.} Throughout, we focus on the regime $M=\Theta(n)$ and $n\to\infty$, where we refer to $\alpha:= M/n$ as the \emph{clause density}. 

Arguably the most fundamental question about random $k$-SAT is \emph{satisfiability}: when does a satisfying assignment, i.e., a $\bs\in\{0,1\}^n$ with $\Phi(\bs)=1$, exist? For fixed $k\in \N$, Franco and Paull~\cite{franco1983probabilistic} showed, using a simple first moment argument, that a random $k$-SAT formula is w.h.p.\,unsatisfiable if $\alpha\ge 2^k\ln 2$. This was later refined by Kirousis et al.~\cite{kirousis1998approximating},
 who established a sharper bound 
 $\alpha\ge 2^k\ln 2-\frac12(\ln 2+1)+o_k(1)$, where the $o_k(1)\to 0$ as $k\to\infty$. On the algorithmic side, Chv{\'a}tal and Reed~\cite{chvatalreed} showed that a simple algorithm known as unit clause w.h.p.\,finds a satisfying assignment if $\alpha<2^k/k$; see also~\cite{ming1990probabilistic}, who showed that the same algorithm succeeds with constant probability. The asymptotic gap between $2^k/k$ and $\Omega_k(2^k)$ was substantially narrowed by Achlioptas and Moore~\cite{achlioptas2002asymptotic}. Using a non-constructive second moment argument, they showed that a random $k$-SAT formula is w.h.p.\,satisfiable if $\alpha \le 2^{k-1}\ln 2-O_k(1)$. Coja-Oghlan and and Panagiotou~\cite{cojapanagio} later refined this bound to $\alpha\le 2^k\ln 2-\frac12(\ln 2+1)+o_k(1)$, nearly matching the lower bound of~\cite{kirousis1998approximating} up to $o_k(1)$ terms. A major breakthrough by Ding, Sly, and Sun~\cite{ding2015proof} confirmed, for large $k$, that the satisfiability threshold is the value predicted by M{\'e}zard, Parisi and Zecchina~\cite{mezard2002analytic}: 
for any large $k$, there exists a threshold $\alpha_*(k)$ such that a random $k$-SAT formula is satisfiable for $\alpha<\alpha_*(k)$ and unsatisfiable for $\alpha>\alpha_*(k)$, both w.h.p.\,as $n\to\infty$. For a detailed discussion on the random $k$-SAT, see the introduction of~\cite{bresler2021algorithmic} and the survey by~\cite{achlioptas2009random}.
\paragraph{Finding a Satisfying Assignment Efficiently} Given the existence of satisfying assignments for $\alpha=O_k(2^k)$, a natural algorithmic question arises: can we find such an assignment efficiently? As mentioned earlier, the simple unit clause algorithm finds a satisfying assignment in polynomial time for $\alpha\le 2^k/k$. Later,  Coja-Oghlan~\cite{coja2010better} devised a polynomial-time algorithm that succeeds for $\alpha\le (1+o_k(1))2^k\ln k/k$, beyond which no efficient algorithm is known. Interestingly, the threshold $(1+o_k(1))2^k\ln k/k$ also marks the onset of \emph{clustering} in the solution space of random $k$-SAT, as predicted by Krzakala et al.~\cite{krzakala2007gibbs} and rigorously verified by Achlioptas and Coja-Oghlan~\cite{achlioptas2008algorithmic}. Beyond this threshold, the solution space of random $k$-SAT fragments into exponentially many well-separated sub-dominant clusters, which are $\Omega(n)$ apart, each containing only an exponentially small fraction of solutions. This phenomenon led to the conjecture that $(1+o_k(1))2^k\ln k/k$ is the algorithmic threshold for the random $k$-SAT.\footnote{It is worth noting that clustering alone does not necessarily imply algorithmic hardness. For instance, the binary perceptron model exhibits an extreme form of clustering at all densities, including the `easy' regime where polynomial-time algorithms exist; see~\cite{perkins2021frozen,abbe2021binary,abbe2021proof,gamarnik2022algorithms,gamarnik2023geometric}.} Specifically, it is conjectured that no polynomial-time algorithm can find (w.h.p.) a satisfying assignment above this threshold. While this conjecture remains open, lower bounds against particular classes of algorithms have been obtained. 
Using the symmetric $m$-OGP, Gamarnik and Sudan~\cite{gamarnik2017performance} showed that \emph{sequential local algorithms} with a moderately growing number of message passing iterations fail to find a satisfying assignment for a variant of the random $k$-SAT model known as the random Not-All-Equal-$k$-SAT (NAE-$k$-SAT) for $\alpha\ge 2^{k-1}\log^2 k/k$. Hetterich~\cite{hetterich2016analysing} proved that \emph{Survey Propagation Guided Decimation} algorithm fails for $\alpha\ge (1+o_k(1))2^k\log k/k$ even with an unbounded number of iterations.  Coja-Oglan, Haqshenas, and Hetterich~\cite{coja2017walksat} showed that \emph{Walksat} stalls for $\alpha\ge C\cdot 2^k\log^2 k/k$, for some absolute constant $C>0$. More recently, Bresler and Huang~\cite{bresler2021algorithmic} demonstrated that \emph{low-degree polynomials} fail to find a satisfying assignment for the random $k$-SAT when $\alpha\ge\kappa^*(1+o_k(1))2^k\ln k/k$, where $\kappa^*\approx 4.911$. Their approach leverages a novel asymmetric variant of the $m$-OGP discussed in Section~\ref{sec:OGP-background}.
%---a powerful framework encompassing many of the aforementioned algorithms---fail to find a satisfying assignment for the random $k$-SAT when $\alpha\ge\kappa^*(1+o_k(1))2^k\ln k/k$, where $\kappa^*\approx 4.911$. Their approach leverages a novel asymmetric variant of the $m$-OGP, discussed in Section~\ref{sec:OGP-background}.
\paragraph{Significance of Symmetric $m$-OGP} 
Our focus is on the symmetric $m$-OGP, which plays a central role in establishing unconditional algorithmic lower bounds. It yields the sharpest known bounds in central models such as the binary perceptron~\cite{gamarnik2022algorithms,li2024discrepancy} and discrepancy minimization~\cite{gamarnik2023geometric}, as well as nearly sharp bounds for NAE-$k$-SAT, a closely related variant of random $k$-SAT~\cite{gamarnik2017performance}. In Ising $p$-spin glasses, symmetric $m$-OGP is asymptotically tight as $p$ grows: its onset matches the algorithmic threshold of the Random Energy Model, the formal $p\to\infty$ limit of
the Ising $p$-spin glasses~\cite{derrida1980random,talagrand2000rigorous,addario2020algorithmic,gamarnik2023shattering}. For fixed $p$, the sharpest lower bounds follow from a very elegant branching OGP argument~\cite{huang2021tight}; however, this  applies only to overlap-concentrated algorithms---a strict subclass of stable algorithms that excludes, for example, low-degree polynomials. In contrast, symmetric $m$-OGP applies to this broader class, albeit yielding weaker bounds. %Although not tight, symmetric $m$-OGP still yields lower bounds against this broader class.
Moreover, branching OGP yields lower bounds only for even $p$, whereas symmetric $m$-OGP applies to both even and odd $p$. In fact, for odd $p$, it provides the only known algorithmic lower bounds for Ising $p$-spin glasses~\cite{gamarnik2023shattering}.
%in fact, the only known algorithmic lower bounds for odd $p$ in Ising spin glasses follow from symmetric $m$-OGP~\cite{gamarnik2023shattering}. 
Further underscoring its relevance, the case $m=2$, or pairwise OGP, is directly linked to the shattering phenomenon---a notion from spin glass theory that has shaped our understanding of algorithmic phase transitions in random CSPs~\cite{achlioptas2006solution,achlioptas2008algorithmic,achlioptas2011solution}, see below. By contrast, such a connection is less clear for more sophisticated OGP variants. Taken together, these observations underscore the significance of symmetric $m$-OGP and highlight the need to fully understand its power.
%Further underscoring its relevance, the special case $m=2$, dubbed pairwise OGP, is directly linked to the shattering phenomenon---a notion emerging from spin glass theory that has shaped our understanding of algorithmic phase transitions in random CSPs~\cite{achlioptas2006solution,achlioptas2008algorithmic,achlioptas2011solution} (see below). A similar link is less clear for more sophisticated variants of the OGP. Collectively, these underscore the significance of symmetric $m$-OGP and highlight the need to fully understand its power.

For $m\in\mathbb{N}$, $\gamma>\gamma_{\mathrm{p-spin}}(m):=1/\sqrt{m}$, and sufficiently large $p$, Ising $p$-spin glasses exhibit symmetric $m$-OGP above the threshold $\gamma\sqrt{2\ln 2}$~\cite[Theorem~2.11]{gamarnik2023shattering}.\footnote{The value $\sqrt{2\ln 2}$ corresponds to the ground-state energy of the Random Energy Model~\cite{talagrand2000rigorous,gamarnik2023shattering}.}: no tuples of equidistant solutions at a prescribed distance exist beyond this threshold (see Definition~\ref{def:m-ogp}). Likewise, for $m\in\mathbb{N}$, $\gamma>\gamma_{\mathrm{k-SAT}}(m):=1/m$, and sufficiently large $k$, the random $k$-SAT model exhibits symmetric $m$-OGP above the constraint density $\gamma 2^k\ln 2$; see~\cite{gamarnik2017performance} or Theorem~\ref{thm:m-ogp-k-sat} below. That is, no tuples of nearly equidistant satisfying assignments at a prescribed distance exist beyond  the density $\gamma 2^k\ln 2$ (see Definition~\ref{def:m-ogp-k-sat}). %For both models, symmetric $m$-OGP rigorously rules out a broad class of algorithms. For further background on the OGP, refer to Section~\ref{sec:OGP-background}.
%Symmetric $m$-OGP rigorously rules out a broad class of stable algorithms for both models~\cite{gamarnik2017performance,bresler2021algorithmic,gamarnik2020low,huang2021tight,gamarnikjagannath2021overlap,gamarnik2023shattering}. For further 
%\textcolor{blue}{Symmetric $m$-OGP plays a central role in establishing algorithmic lower bounds. It yields the sharpest known bounds in central models such as binary perceptron~\cite{gamarnik2022algorithms,li2024discrepancy} and discrepancy minimization~\cite{gamarnik2023geometric}, as well as nearly sharp bounds for NAE-$k$-SAT, a closely related variant of random $k$-SAT~\cite{gamarnik2017performance}. Moreover, while the sharpest lower bounds for Ising spin glasses have been obtained through branching OGP~\cite{huang2021tight}, this approach applies only to overlap-concentrated algorithms---a strict subclass of stable algorithms that excludes, e.g., low-degree polynomials. Although not tight, symmetric $m$-OGP yields lower bounds against this broader class. Notably, branching OGP applies to even $p$. To the best of our knowledge, symmetric $m$-OGP provides the only known algorithmic lower bounds for odd $p$ in Ising spin glasses~\cite{gamarnik2023shattering}. Taken together, these highlight the significance of symmetric $m$-OGP and underscore the importance of understanding its full power.}
\subsection{Overview of Our Results}\label{sec:Overview}
As noted, studying the absence of OGP is essential for a deeper understanding of OGP-based methods and for developing a more complete average-case theory, yet this question has not been rigorously addressed. We initiate this direction by investigating the absence of symmetric $m$-OGP for Ising $p$-spin glass and random $k$-SAT for large $p$ and $k$---two fundamental average-case models. This is a natural starting point: extending such investigations to other variants of OGP (e.g., asymmetric OGP~\cite{bresler2021algorithmic} or branching OGP~\cite{huang2021tight}) appears far more challenging, and symmetric OGP provides a natural foundation for these extensions. Moreover, both Ising $p$-spin and random $k$-SAT are more tractable for large $p$ and $k$. Rigorous results for small, fixed $k$ remain elusive for random $k$-SAT, and results for fixed $p$ are extremely challenging for spin glasses. This reflects a broad trend in random structures; see Remark~\ref{remark:Large-p-k}.

Prior work, focusing exclusively on leveraging symmetric $m$-OGP to establish algorithmic hardness,  overlooked the case $\gamma<\gamma(m)$, where $\gamma_{\mathrm{p-spin}}(m)=1/\sqrt{m}$ or $\gamma_{\mathrm{k-SAT}}(m)=1/m$. Our first contribution is to precisely determine the thresholds for both models, providing deeper insights into
the power of symmetric $m$-OGP. All statements below hold w.h.p.\,as $n\to\infty$. We begin with the Ising spin glasses.
\begin{theorem}{(Informal, see Theorem~\ref{thm:m-ogp-tight})}\label{thm:absent-informal-spin}
For any $m\in\mathbb{N}$, $\gamma<\gamma_{\mathrm{p-spin}}(m)=1/\sqrt{m}$ and sufficiently large $p$, the symmetric $m$-OGP is absent in the Ising $p$-spin glass below the value $\gamma\sqrt{2\ln 2}$.
   % Fix any $m\in\N$ and any $\gamma<1/\sqrt{m}$. Then for all large $p$, the symmetric $m$-OGP is provably absent in the Ising $p$-spin model below the threshold 
\end{theorem}
Namely, for any $m\in\mathbb{N}$ and sufficiently large $p$, nearly equidistant $m$-tuples of solutions exist at all pairwise distances below the value $\sqrt{(2\ln 2)/m}$. Theorem~\ref{thm:absent-informal-spin} offers a qualitative insight into the OGP-based arguments. Specifically, in the limit of large $p$, it reveals that the power of OGP in establishing algorithmic hardness indeed strictly improves with $m$. We next address the random $k$-SAT.
\begin{theorem}{(Informal, see Theorem~\ref{thm:m-ogp-k-sat-absent})}\label{thm:absent-informal-sat}
For any $m\in\N$, $\gamma<\gamma_{\mathrm{k-SAT}}(m)=1/m$ and $k=\Omega(\ln n)$, the symmetric $m$-OGP is absent in the random $k$-SAT below the density $\gamma 2^k\ln 2$.
\end{theorem}
Theorem~\ref{thm:absent-informal-sat} asserts that for any $m\in\mathbb{N}$ and mildy growing $k$, $k=\Omega(\ln n)$, nearly equidistant $m$-tuples of satisfying assignments exist at all pairwise distances below the constraint density $(2^k\ln 2)/m$. Recall that the algorithmic threshold for the random $k$-SAT is asymptotic to $2^k\ln 2/k$ as $k\to\infty$. Consequently, Theorem~\ref{thm:absent-informal-k-sat} implies that achieving matching algorithmic lower bounds via the symmetric $m$-OGP requires $m$ to be superconstant. Indeed, sharp lower bounds for this model were obtained instead through an asymmetric variant of OGP~\cite{bresler2021algorithmic}. %(Symmetric OGP does give nearly sharp bounds for NAE-$k$-SAT.)

Theorem~\ref{thm:absent-informal-sat} holds under the assumption that $k$ grows mildly in $n$. For constant $k$, we establish: %Some of the earlier developments in the random $k$-SAT literature have also considered the regime $k=\Omega(\ln n)$, see~\cite{frieze2005random,coja2008random,liu2012note}. 

\begin{theorem}{(Informal, see Theorem~\ref{thm:m-ogp-k-sat-absent-constant-k})}\label{thm:absent-informal-k-sat}
    For any $m\in\mathbb{N}$, $\gamma<1/m$ and sufficiently large $k$, the symmetric $m$-OGP is absent  below the density $\gamma 2^k\ln 2$ for the set of assignments that satisfy a $1-\exp(-\Theta_k(k))$ fraction of clauses.
\end{theorem}
Namely, the $m$-OGP is absent for the set of assignments that violate a small fraction of clauses.

Combined with prior results on the $m$-OGP, Theorems~\ref{thm:absent-informal-spin}-\ref{thm:absent-informal-sat} establish a sharp phase transition.
%Theorems~\ref{thm:absent-informal-spin}-\ref{thm:absent-informal-sat} justify the aforementioned practice of increasing $m$ when applying OGP. Moreover, combined with the prior results on the $m$-OGP, they establish a sharp phase transition in both models.
\begin{theorem}{(Informal, see Theorems~\ref{thm:m-ogp-sharp-PT} and~\ref{thm:m-ogp-sharp-PT-sat})}\label{thm:PTT}
    For any $m\in\N$, the symmetric $m$-OGP undergoes a sharp phase transition in both the Ising $p$-spin glass and the random $k$-SAT models.
\end{theorem}
Phase transition points are $\gamma_{\mathrm{p-spin}}=1/\sqrt{m}$ and $\gamma_{\mathrm{k-SAT}}=1/m$, both of which are strictly monotonic in $m$. While the first moment calculations had suggested that the OGP thresholds improve with $m$---as expected, since $m$-OGP structure becomes increasingly nested as $m$ increases---a rigorous confirmation had been lacking. Theorem~\ref{thm:PTT} establishes this structural fact for the models we study, thereby justifying the aforementioned common practice of increasing $m$ when applying OGP-based methods.
\begin{remark}\label{remark:Large-p-k}
Notice that Theorems~\ref{thm:absent-informal-spin}-\ref{thm:PTT} hold for large $p$ and $k$. This mirrors a rich body of work in random structures---e.g., random $k$-SAT, spin glasses, and random graphs---where rigorous results for fixed parameters are notoriously challenging and foundational progress has largely occurred in the large parameter regime. For random $k$-SAT, both the satisfiability threshold~\cite{ding2015proof} and algorithmic lower bounds~\cite{bresler2021algorithmic} are known only for sufficiently large $k$. Earlier developments~\cite{frieze2005random,coja2008random,liu2012note} focused on the same regime $k=\Omega(\ln n)$ considered here. For Ising $p$-spin glasses, rigorous results for large $p$~\cite{talagrand2000rigorous} preceded technically far more delicate breakthroughs for fixed $p$~\cite{talagrand2006parisi,panchenko2013parisi}. For sparse random graphs, most rigorous results pertain to large average degree $d$~\cite{frieze1990independence,bayati2010combinatorial,wein2020optimal}. The large $p$ regime has also proven useful in high-dimensional statistics; for instance, a recent work rigorously analyzes the large average subtensor problem in a challenging regime for $p$ large~\cite{kizildag2025tensor}.
 Our work, which initiates a new direction into the absence of OGP, also focuses on large parameter regime; extending beyond this, particularly to fixed $p$ or $k$, is a central open challenge for future work. 
%Extending beyond this regime (to fixed $p,k$) is a central open challenge for future work.
%Our work, initiating the study of the absence of OGP, also addresses the large parameter regime; extension beyond remains an open direction for future work.
\end{remark}
\paragraph{Shattering in Spin Glasses} The celebrated work by Kirkpatrick and Thirumalai~\cite{kirkpatrick1987p} predicted a shattering phase in the Gibbs measure of the Ising $p$-spin glass, in which exponentially many well-separated clusters, each with exponentially small mass, collectively contain almost all of the Gibbs mass. Subsequent work~\cite{montanari2003nature,ferrari2012two} refined this conjecture, postulating shattering over a range of inverse temperatures $\beta\in (\beta_d(p),\beta_c(p))$, where $\beta_d(p) = (1+o_p(1))\sqrt{(2\ln p)/p}$ as shown in~\cite{ferrari2012two,alaoui2023sampling} and  $\beta_c(p)=(1-o_p(1))\sqrt{2\ln 2}$ as established in~\cite{talagrand2000rigorous}, both asymptotically as $p\to\infty$. Shattering at zero temperature (i.e., shattering of the solution space)
has played a key role since then in the study of algorithmic hardness in random CSPs~\cite{achlioptas2006solution,achlioptas2008algorithmic,achlioptas2011solution}. 

Despite major advances in spin glass theory, shattering phenomena have only recently begun to be rigorously understood. \cite{gamarnik2023shattering} established shattering in the Ising $p$-spin glass model for $\beta\in(\sqrt{\ln 2},\sqrt{2\ln 2})$ and all sufficiently large $p$. Their argument, based on the pairwise OGP (2-OGP), falls short of establishing shattering when $\beta_d(p)<\beta<\sqrt{\ln 2}$. In light of Theorems~\ref{thm:absent-informal-spin} and~\ref{thm:m-ogp-tight}, the 2-OGP is absent when $\beta<\sqrt{\ln 2}$, which accounts for the fact that the techniques of~\cite{gamarnik2023shattering} apply only to the range $\sqrt{\ln 2}<\beta<\sqrt{2\ln 2}$. Very recently, El Alaoui~\cite{alaoui2024near} introduced a `soft' version of the 2-OGP, which led to a near-optimal shattering result extending down to the conjectured $\beta_d(p)$ threshold. Our results confirm that the bounds in~\cite{gamarnik2023shattering} are tight. This highlights the necessity of the soft-OGP approach developed in~\cite{alaoui2024near} to obtain  optimal shattering results.
%which, in turn, highlights the necessity of the soft-OGP approach .
%\newpage
%\textcolor{red}{Proofs of Theorems~\ref{thm:absent-informal-spin}-\ref{thm:absent-informal-sat} crucially rely on an involved application of the \emph{second moment method}. The direct application of the second moment method fails for Theorems~\ref{thm:absent-informal-spin}-\ref{thm:absent-informal-sat}: it yields the existence of a certain $m$-tuple of solutions with probability exponentially small in $n$ with an improving exponent as $p,k\to\infty$. Nevertheless, combined with a concentration property applied to a suitable random variable, the second moment method can be `repaired' using a very elegant argument due to ~\cit1990independence}.}

\vspace{0.05in}
\noindent{\bf Notation} %For any set $A$, let $|A|$ denotes its cardinality. 
For any event $E$, $\ind\{E\}$ denotes its indicator. For $r\in\mathbb{R}$, $\exp(r):=e^r$. 
For $n\in\N$, $\Sigma_n:= \{-1,1\}^n$. For any $\bs,\bs'\in\Sigma_n$ or $\bs,\bs'\in\{0,1\}^n$, $d_H(\bs,\bs'):=\sum_{1\le i\le n}\ind\{\bs(i)\ne \bs'(i)\}$. For $\bs,\bs'\in\Sigma_n$, $\ip{\bs}{\bs'}:= \sum_{1\le i\le n}\bs(i)\bs'(i)=n-2d_H(\bs,\bs')$. %Given any $\boldsymbol{\mu}\in\R^n$ and positive semidefinite $\boldsymbol{\Sigma}\in\R^{n\times n}$, denote by $\cN(\boldsymbol{\mu},\boldsymbol{\Sigma})$ the multivariate normal distribution on $\R^n$ with mean $\boldsymbol{\mu}$ and covariance $\boldsymbol{\Sigma}$. For any matrix $\mathcal{M}$, $\|\mathcal{M}\|_F$, $\|\mathcal{M}\|_2$, and $|\mathcal{M}|$ respectively denote its Frobenius norm, spectral norm, and determinant. 
We employ standard asymptotic notation: $\Theta(\cdot),O(\cdot),o(\cdot),\Omega(\cdot)$, and $\omega(\cdot)$; the underlying asymptotics is often w.r.t.\,$n\to\infty$. 
Asymptotics other than $n\to\infty$ are reflected with a subscript, e.g.\,$\Omega_k(\cdot)$. % All floor/ceiling operators are omitted for simplicity.
%\paragraph{Paper Organization} The rest of the paper is organized as follows. We provide our results regarding the $p$-spin model in Section~\ref{sec:SPIN} and the $k$-SAT model in Section~\ref{sec:SAT}. We address the case of constant $k$ in  Section~\ref{set:CONSTANT}. Complete proofs of all our results are provided in Section~\ref{sec:PFs}.
\section{Sharp Phase Transition for the $m$-OGP: Ising $p$-Spin Glass Model}\label{sec:SPIN}
In this section, we establish a sharp phase transition for the multi Overlap Gap Property ($m$-OGP) in the Ising $p$-spin model. We first formalize the set of $m$-tuples under consideration. 
\begin{definition}\label{def:m-ogp}
    Let $m\in\N$, $0<\gamma<1$, and $0<\eta<\xi<1$. Denote by $\mathcal{S}_{\mathrm{p-spin}}(\gamma,m,\xi,\eta)$ the set of all $m$-tuples $\bs^{(t)}\in\Sigma_n:=\{-1,1\}^n$, $1\le t\le m$, that satisfy the following:
    \begin{itemize}
        \item {\textbf{$\gamma$-Optimality:}} For any $1\le t\le m$, $H(\bs^{(t)})\ge \gamma\sqrt{2\ln 2}$.
        \item {\textbf{Overlap Constraint:}} For any $1\le t<\ell \le m$, $n^{-1}\ip{\bs^{(t)}}{\bs^{(\ell)}}\in[\xi-\eta,\xi]$.
    \end{itemize}
\end{definition}
Definition~\ref{def:m-ogp} regards $m$-tuples of near-optimal solutions whose (pairwise) overlaps are constrained; parameter $\gamma$ quantifies the near-optimality and $\xi,\eta$ control the region of overlaps. The model exhibits symmetric $m$-OGP at threshold $\gamma\sqrt{2\ln 2}$ if $\mathcal{S}_{\mathrm{p-spin}}$ is w.h.p.\,empty for suitable $m,\xi$ and $\eta$.

Gamarnik, Jagannath, and K{\i}z{\i}lda\u{g} establish the following result.
\begin{proposition}{\cite[Theorem~2.11]{gamarnik2023shattering}}\label{thm:m-ogp}
    For any $m\in\N$ and any $\gamma>1/\sqrt{m}$, there exists $0<\eta<\xi<1$ and a $P^*\in\N$ such that the following holds. Fix any $p\ge P^*$. Then,  as $n\to\infty$
    \[ \mathbb{P}\bigl[\mathcal{S}_{\mathrm{p-spin}}(\gamma,m,\xi,\eta)\ne\varnothing\bigr]\le e^{-\Theta(n)}.
    \]
\end{proposition}
That is, for any $m\in\N$ and $\gamma>1/\sqrt{m}$, Ising pure $p$-spin model exhibits the symmetric $m$-OGP above $\gamma\sqrt{2\ln 2}$ for all large enough $p$.\footnote{We emphasize that \cite[Theorem~2.11]{gamarnik2023shattering} is, in fact, stronger. It establishes the \emph{ensemble $m$-OGP}, concerning tuples of near-optima w.r.t.\,correlated Hamiltonians, which is necessary for proving algorithmic hardness.} 

Proposition~\ref{thm:m-ogp} focuses on $\gamma>1\sqrt{m}$. Our first main result addresses the regime
$\gamma<1/\sqrt{m}$.
\begin{theorem}\label{thm:m-ogp-tight}
    For any $m\in \N$, $\gamma<1/\sqrt{m}$, and any $0<\eta<\xi<1$, there exists a $P^*\in \N$ such that the following holds. Fix any $p\ge P^*$. Then, as $n\to\infty$
    \[
\mathbb{P}\bigl[\mathcal{S}_{\mathrm{p-spin}}(\gamma,m,\xi,\eta)\ne\varnothing\bigr]\ge 1-e^{-\Theta(n)}. 
    \]
\end{theorem}
The proof of Theorem~\ref{thm:m-ogp-tight} relies on multiple techniques, which we overview in  Section~\ref{sec:pf_technique}. See Section~\ref{sec:pf-m-ogp-tight} for the full proof.

Taken together, Proposition~\ref{thm:m-ogp} and Theorem~\ref{thm:m-ogp-tight} collectively imply that for sufficiently large $p$, the value $\sqrt{(2\ln 2)/m}$ is tight: for any $m\in\N$, and sufficiently large $p$, the onset of $m$-OGP is $\sqrt{(2\ln 2)/m}$. We 
make this more precise in Theorem~\ref{thm:m-ogp-sharp-PT}, establishing a sharp phase transition.

We emphasize that Theorem~\ref{thm:m-ogp-tight} regards a single Hamiltonian, rather than a correlated ensemble. Crucially, it holds for \emph{any} $\xi$ and $\eta$, strengthening the ensuing results: the phase transition for the $m$-OGP is independent of $\xi,\eta$ or the ensemble, making it inherently definition-robust.
\paragraph{Sharp Phase Transition} We define the notion of \emph{admissible} values of $\gamma$.
\begin{definition}\label{def:adm-exp}
    Fix an $m\in\N$. A value $\gamma>0$ is called $m$-admissible if for any $0<\eta<\xi<1$ there exists a $P^*(m,\gamma,\eta,\xi)\in\N$ such that for every fixed $p\ge P^*(m,\gamma,\eta,\xi)$,
    \[
\lim_{n\to\infty}\mathbb{P}\bigl[\mathcal{S}_{\mathrm{p-spin}}(\gamma,m,\xi,\eta)\ne\varnothing\bigr] =1.
    \]
    Similarly, $\gamma$ is called $m$-inadmissible if there exists $0<\eta<\xi<1$, $\eta$ sufficiently small, and a $\hat{P}\in \N$ such that for every fixed $p\ge \hat{P}$,
    \[
\lim_{n\to\infty}\mathbb{P}\bigl[\mathcal{S}_{\mathrm{p-spin}}(\gamma,m,\xi,\eta)\ne\varnothing\bigr] =0.
    \]
\end{definition}
The sharp phase transition we investigate is formalized as follows.
\begin{definition}\label{def:sharp-PT}
    Fix $m\in\N$. The $m$-OGP exhibits a sharp phase transition at the threshold $\gamma(m)$ if
    \[
    \sup\{\gamma>0:\gamma\text{ is $m$-admissible}\} = \gamma(m)=\inf\{\gamma>0:\gamma\text{ is $m$-inadmissible}\}.
    \]
\end{definition}
As noted above, the phase transition is definition-robust, i.e., it remains independent of $\xi,\eta$ or the ensemble. We establish that: \begin{theorem}\label{thm:m-ogp-sharp-PT}
    For any $m\in\N$, the $m$-OGP for the Ising $p$-spin glass model exhibits a sharp phase transition in the sense of Definition~\ref{def:sharp-PT} at $\gamma_{\mathrm{p-spin}}(m)=1/\sqrt{m}$.
\end{theorem}
Theorem~\ref{thm:m-ogp-sharp-PT} follows directly from Proposition~\ref{thm:m-ogp} and Theorem~\ref{thm:m-ogp-tight}; the proof is omitted.  We highlight that the phase transition point $\gamma_{\mathrm{p-spin}}(m)$ is strictly monotonic in $m$.

\subsection{{Proof Overview for Theorem~\ref{thm:m-ogp-tight}}}\label{sec:pf_technique}
Given $m\in\mathbb{N}$ and $0<\eta<\xi<1$, denote by $\mathcal{F}(m,\xi,\eta)$ the set of all $(\bs^{(1)},\dots,\bs^{(m)})$, satisfying the overlap constraints, $n^{-1}\langle \bs^{(i)},\bs^{(j)}\rangle \in[\xi-\eta,\xi]$ for all $i\ne j$. To begin with, we must show that $\mathcal{F}(m,\xi,\eta)$ is non-empty; otherwise the set $\mathcal{S}_{\mathrm{p-spin}}$ in Definition~\ref{def:m-ogp} would be trivially empty. We establish this using \emph{probabilistic method}~\cite{alon2016probabilistic}. Specifically, we generate $\bs^{(1)},\dots,\bs^{(m)}$ by independently sampling each coordinate $\bs^{(i)}(k)\in\{-1,1\}$ with an appropriate mean. Applying large deviations inequalities, we show that the resulting tuple satisfies the overlap constraints with positive probability, thereby ensuring $\mathcal{F}(m,\xi,\eta)\ne\varnothing$. 

The next step is applying the second moment method. Suppose that $M$ is a non-negative integer-valued random variable with finite variance. Then, $\mathbb{P}[M\ge 1]\ge \mathbb{E}[M]^2/\mathbb{E}[M^2]$ by the Paley-Zygmund Inequality. We fix $\gamma<1/\sqrt{m}$ and apply this to $M:=|\mathcal{S}_{\mathrm{p-spin}}(\gamma,m,\xi,\eta)|$ in Proposition~\ref{prop:2nd-moment}. To estimate $\mathbb{E}[M^2]$, we analyze a summation over \emph{all pairs} of $m$-tuples in $\mathcal{F}(m,\xi,\eta)$ satisfying Definition~\ref{def:m-ogp}. Building on an overcounting argument, we show that for large $p$, these pairs behave approximately independently unless their corresponding coordinates exhibit extreme correlation. At the same time, enforcing a high degree of correlation incurs a steep entropic cost. This allows us to decompose $\mathbb{E}[M^2]$ into two parts, each of which can be controlled separately. This leads to 
\begin{equation}\label{eq:2NDMOM}
    \mathbb{P}\bigl[\bigl|\mathcal{S}_{\mathrm{p-spin}}(\gamma,m,\xi,\eta)\bigr|\ge 1\bigr]\ge \exp\bigl(-no_p(1)\bigr)
\end{equation}
where $o_p(1)\to 0$ as $p\to\infty$. Thus, the second moment method alone falls short of establishing the desired result that $\mathcal{S}_{\mathrm{p-spin}}\ne\varnothing$ w.p.\,$1-o(1)$. While this issue could be resolved for large $p$ %As we describe below, the second moment bound could be fixed for large $p$, 
if $|\mathcal{S}_{\mathrm{p-spin}}|$ were concentrated around its mean, such a concentration property is unclear. Instead, we proceed  by defining a suitable proxy random variable that is closely related to $|\mathcal{S}_{\mathrm{p-spin}}|$ and more tractable:
\[
T_{m,\xi,\eta}:=\max_{(\bs^{(1)},\dots,\bs^{(m)})\in \mathcal{F}(m,\xi,\eta)} \min_{1\le j\le m}H(\bs^{(j)}).
\]
In Lemma~\ref{lemma:T-concen}, we prove that $T_{m,\xi,\eta}$, viewed as a function of the tensor $\boldsymbol{J}$ with i.i.d.\,$\cN(0,1)$ entries (cf.~\eqref{eq:Hamiltonian-p-spin}), satisfies a Lipschitz property. Consequently, standard concentration results for Lipschitz functions of normal random variables~\cite[Theorem~2.26]{wainwright2019high} yield that for all $\epsilon\ge 0$:
\begin{equation}\label{eq:T-CON}
    \mathbb{P}\bigl[\bigl|T_{m,\xi,\eta}-\mathbb{E}[T_{m,\xi,\eta}]\bigr|\ge \epsilon\bigr]\le 2\exp\bigl(-n\epsilon^2/2\bigr).
\end{equation}
\paragraph{Second Moment Bound for Large $p$} We now leverage an elegant argument of Frieze~\cite{frieze1990independence} to combine the second moment bound with concentration. The key observation is that:
\begin{equation}\label{eq:CRUCIALE}
 |\mathcal{S}_{\mathrm{p-spin}}(\gamma,m,\xi,\eta)|\ge 1 \iff T_{m,\xi,\eta}\ge \gamma \sqrt{2\ln 2},
 \end{equation}
 for any $\gamma>0$. Fixing $\gamma'\in(\gamma,1/\sqrt{m})$, we obtain that for all sufficiently large $p$ and $n$,
\[
\mathbb{P}[T_{m,\xi,\eta}\ge\gamma'\sqrt{2\ln 2}]\ge \exp\bigl(-no_p(1)\bigr)\ge 2\exp(-n\epsilon^2/2)\ge
\mathbb{P}\bigl[T_{m,\xi,\eta}\ge \mathbb{E}[T_{m,\xi,\eta}]+\epsilon\bigr]
\]
using~\eqref{eq:2NDMOM},~\eqref{eq:T-CON}, and~\eqref{eq:CRUCIALE}. Consequently, $\mathbb{E}[T_{m,\xi,\eta}]\ge \gamma'\sqrt{2\ln 2}-\epsilon$. Applying~\eqref{eq:T-CON} once more, we obtain w.h.p.\,$T_{m,\xi,\eta}\ge \mathbb{E}[T_{m,\xi,\eta}]-\epsilon\ge \gamma'\sqrt{2\ln 2}-2\epsilon$, which at least $\gamma\sqrt{2\ln 2}$ provided $\epsilon>0$ is sufficiently small.
This bound, together with~\eqref{eq:CRUCIALE}, establishes Theorem~\ref{thm:m-ogp-tight}.
\section{Sharp Phase Transition for the $m$-OGP: Random $k$-SAT Model}\label{sec:SAT}
We now turn to the random $k$-SAT. The set of $m$-tuples we investigate is as follows.
\begin{definition}\label{def:m-ogp-k-sat}
    Let $k\in \N$, $\gamma\in (0,1)$, $m\in \N$, and $0<\eta<\xi<1$. Denote by $\mathcal{S}_{{\mathrm{k-SAT}}}(\gamma,m,\xi,\eta)$ the set of all $m$-tuples $\bs^{(t)}\in\{0,1\}^n$, $1\le t\le m$, that satisfy the following.
    \begin{itemize}
        \item {\textbf{Satisfiability:}} For any $1\le t\le m$, $\Phi(\bs^{(t)})=1$, where $\Phi(\cdot)$ is a conjunction of $n\alpha_\gamma$ independent $k$-clauses and $\alpha_\gamma = \gamma 2^k\ln 2$.
        \item {\textbf{Pairwise  Distance:}} For any $1\le t<\ell \le m$, $n^{-1}d_H(\bs^{(t)},\bs^{(\ell)})\in[\xi-\eta,\xi]$.
    \end{itemize}
\end{definition}
Definition~\ref{def:m-ogp-k-sat} regards $m$-tuples of satisfying solutions whose pairwise distances are constrained. The parameter $\alpha_\gamma$ is the clause density: the formula $\Phi$ consists of $n\alpha_\gamma$  independent $k$-clauses. Parameters $\xi$ and $\eta$ collectively control the region of overlaps. Similar to above, the model exhibits $m$-OGP above densities $\alpha_\gamma$ if $\mathcal{S}_{\mathrm{k-SAT}}$ w.h.p.\,empty for suitable $m,\xi$ and $\eta$.

Gamarnik and Sudan~\cite{gamarnik2017performance} established the symmetric $m$-OGP for a variant of random $k$-SAT known as NAE-$k$-SAT. We first extend their result to random $k$-SAT model.
\begin{theorem}\label{thm:m-ogp-k-sat}
    For any $m\in\N$ and any $\gamma>1/m$, there exists $0<\eta<\xi\le\frac12$ and $K^*:= K^*(\gamma,m,\xi,\eta)\in \N$ such that the following holds. Fix any $k\ge K^*$. Then, as $n\to\infty$
    \[
    \mathbb{P}\bigl[\mathcal{S}_{{\mathrm{k-SAT}}}(\gamma,m,\xi,\eta)\ne\varnothing\bigr]\le e^{-\Theta(n)}.
    \]
\end{theorem}
Theorem~\ref{thm:m-ogp-k-sat} is shown via the first moment method, see Section~\ref{sec:pf-ogp-k-sat} for the proof. For any $m\in\N$ and any $\gamma>1/m$, the random $k$-SAT model exhibits $m$-OGP at clause density $\gamma 2^k\ln 2$. 

Our next result addresses the regime $\gamma<1/m$.
\begin{theorem}\label{thm:m-ogp-k-sat-absent}
    For any $m\in \N$, $\gamma<1/m$ and $0<\eta<\xi\le \frac12$, there exists a constant $C(\gamma,m,\xi,\eta)>0$ such that the following holds. Fix any $k\ge C(\gamma,m,\xi,\eta)\ln n$. Then as $n\to\infty$,
    \[
    \mathbb{P}\bigl[\mathcal{S}_{\mathrm{k-SAT}}(\gamma,m,\xi,\eta)\ne\varnothing\bigr]\ge 1-e^{-\Theta(n)}.
    \]
\end{theorem}
The proof of Theorem~\ref{thm:m-ogp-k-sat-absent} relies on the second moment method, see Section~\ref{sec:pf-ogp-absent-k-sat} for details. Theorem~\ref{thm:m-ogp-k-sat-absent} yields that for any $m\in\N$, $\alpha<(2^k\ln 2)/m$ and $k$ growing mildly in $n$, the $m$-OGP is absent: nearly equidistant $m$-tuples of satisfying solutions exist at all pairwise distances. Taken together, Theorems~\ref{thm:m-ogp-k-sat} and~\ref{thm:m-ogp-k-sat-absent} collectively yield that for $k=\Omega(\ln n)$, the onset of $m$-OGP occurs at density $(2^k\ln 2)/m$. In Theorem~\ref{thm:m-ogp-sharp-PT-sat} below, we combine Theorems~\ref{thm:m-ogp-k-sat} and~\ref{thm:m-ogp-k-sat-absent} to establish a sharp phase transition.
%In the next section, we leverage Theorems~\ref{thm:m-ogp-k-sat} and~\ref{thm:m-ogp-k-sat-absent} to establish a sharp phase transition analogous to Theorem~\ref{thm:m-ogp-sharp-PT}. 

We remark on the condition $k=\Omega(\log n)$. Several earlier developments in the random $k$-SAT literature, particularly those concerning the satisfiability threshold, were also established in the same regime $k=\Omega(\ln n)$, see~\cite{frieze2005random,coja2008random,liu2012note}. In Theorem~\ref{thm:m-ogp-k-sat-absent-constant-k} below, 
we establish the absence of the $m$-OGP for all sufficiently large $k=O(1)$, in the set of nearly satisfying assignments that violate only a small fraction of clauses. 
%We note that unlike our results for the $p$-spin glass model that are valid for sufficiently large $p$, Theorem~\ref{thm:m-ogp-k-sat-absent} holds for $k=\Omega(\ln n)$. Consequently, the corresponding phase transition also holds for $k=\Omega(\ln n)$. Later in Theorem~\ref{thm:m-ogp-k-sat-absent-constant-k}, we provide a weaker version of Theorem~\ref{thm:m-ogp-k-sat-absent} for $m$-tuples of nearly satisfying assignments that holds for all sufficiently large $k=O(1)$. Admittedly, the regime $k=\Omega(\ln n)$ does not truly correspond to the most interesting case $M=\Theta(n)$, where $M$ is the number of clauses. Nevertheless, it is worth mentioning that certain earlier developments in the random $k$-SAT literature, particularly those regarding the satisfiability threshold, were also established in the same regime $k=\Omega(\ln n)$, see e.g.~\5random,coja2008random,liu2012note}.

Our final remark concerns the technical condition $\xi\le \frac12$ in Theorem~\ref{thm:m-ogp-k-sat-absent}. For Theorem~\ref{thm:m-ogp-k-sat-absent} to be non-trivial, there must exist $m$-tuples $\bs^{(1)},\dots,\bs^{(m)}\in\{0,1\}^n$ such that for all sufficiently small $\eta>0$, their Hamming distances satisfy $d_H(\bs^{(i)},\bs^{(j)})\in[\xi-\eta,\xi]$ for $1\le i<j\le m$. (In fact, our proof  extends to any $\xi$ for which such tuples exist.) An application of the probabilistic method---in particular Lemma~\ref{lemma:F-NON-EMPTY}---confirms the existence of such tuples for $\xi\le \frac12$; however, the case $\xi>\frac12$ remains unclear.\footnote{While Lemma~\ref{lemma:F-NON-EMPTY} focuses on $\bs\in\{-1,1\}^n$, a straightforward modification extends it to $\{0,1\}^n$ for $\xi\le \frac12$.}  For this reason, we restrict our attention to $\xi\le \frac12$, leaving $\xi>\frac12$ for future work.
\paragraph{A Sharp Phase Transition} 
Similar to the Ising $p$-spin glass model, we define the notion of \emph{admissible} values of $\gamma$ by modifying Definition~\ref{def:adm-exp}. 
\begin{definition}\label{def:adm-gamma-sat}
    Fix an $m\in\N$. A value $\gamma>0$ is called $m$-admissible if for any $0<\eta<\xi<1$ there exists a $C(\gamma,m,\xi,\eta)\in\N$ such that for every $k\ge C(\gamma,m,\xi,\eta)\ln n$,
    \[
\lim_{n\to\infty}\mathbb{P}\bigl[\mathcal{S}_{\mathrm{k-SAT}}(\gamma,m,\xi,\eta)\ne\varnothing\bigr] =1.
    \]
    Similarly, $\gamma$ is called $m$-inadmissible if there exists $0<\eta<\xi<1$, $\eta$ sufficiently small, and a $C(\gamma,m,\xi,\eta)\in\N$ such that for every $k\ge C(\gamma,m,\xi,\eta)\ln n$,
    \[
\lim_{n\to\infty}\mathbb{P}\bigl[\mathcal{S}_{\mathrm{k-SAT}}(\gamma,m,\xi,\eta)\ne\varnothing\bigr] =0.
    \]
\end{definition}
Equipped with Definition~\ref{def:adm-gamma-sat}, we now analyze the sharp phase transition of the $m$-OGP in the sense of
Definition~\ref{def:sharp-PT} and establish the following result.
\begin{theorem}\label{thm:m-ogp-sharp-PT-sat}
    For any $m\in\N$, the $m$-OGP for the random $k$-SAT model exhibits a sharp phase transition in the sense of Definition~\ref{def:sharp-PT} at $\gamma_{\mathrm{k-SAT}}=1/m$.
\end{theorem}
Theorem~\ref{thm:m-ogp-sharp-PT-sat} follows directly from Theorems~\ref{thm:m-ogp-k-sat} and~\ref{thm:m-ogp-k-sat-absent}; see Section~\ref{pf:PT} for its proof. Notably, the phase transition point $\gamma_{\mathrm{k-SAT}}(m)$ is, once again, strictly monotonic in $m$.

\subsection{Case of Constant $k$}\label{set:CONSTANT}
The $m$-OGP result for the random $k$-SAT (Theorem~\ref{thm:m-ogp-k-sat}) holds for all sufficiently large $k$, while Theorem~\ref{thm:m-ogp-k-sat-absent} along with the sharp phase transition result (Theorem~\ref{thm:m-ogp-sharp-PT-sat}) apply to $k=\Omega(\ln n)$. A fundamental question arises: 
can we extend the analysis to the challenging regime where $k$ is constant? 

In this section, we address this question by analyzing a variant of the $k$-SAT model that focuses on nearly satisfying assignments---assignments that violate only a small fraction of clauses. This setting is closely related to the maximum satisfiability (MAX-SAT) model. For this model, we establish the absence of the $m$-OGP for sufficiently large constant $k$.
%Extension to random $k$-SAT is among the future research directions we highlight in Section~\ref{sec:future-work}.
%When $k$ remains constant as $n\to\infty$, $k=O(1)$, the second moment method fails: it yields $\mathbb{P}[\mathcal{S}_{\mathrm{k-SAT}}(\gamma,m,\xi,\eta)\ne\varnothing]\ge \exp\bigl(-nc^k\bigr)$ for some $c:=c(\gamma,\xi,\eta)\in(0,1)$. Notice that while the probability guarantee is exponentially small in $n$, the exponent improves as $k$ grows. This is an instance where the unsuccessful second moment method can be circumvented by considering an appropriate random variable satisfying a concentration property and using a trick of Frieze~\cite{frieze1990independence}. 
\paragraph{An Objective Function: Number of Violated Constraints} 
Given a random $k$-SAT formula $\Phi$ which is a conjunction of $M$ independent $k$-clauses $\mathcal{C}_1,\dots,\mathcal{C}_M$, and a truth assignment $\bs\in\{0,1\}^n$, define $\mathcal{L}(\bs)$ as the number of clauses violated by $\bs$:
$\mathcal{L}(\bs) = \sum_{1\le i\le M}\ind\bigl\{\mathcal{C}_i(\bs)=0\bigr\}$. In particular, $\Phi(\bs)=1$ iff $\mathcal{L}(\bs)=0$. With this, we modify Definition~\ref{def:m-ogp-k-sat} accordingly.
\begin{definition}\label{def:modified-k-sat}
    Let $k\in \N$, $\gamma\in (0,1)$, $m\in \N$, $0<\eta<\xi<1$, and $\kappa\in[0,1]$. Denote by $\mathcal{S}_{\mathrm{k-SAT}}(\gamma,m,\xi,\eta,\kappa)$ the set of all $m$-tuples $\bs^{(t)}\in\{0,1\}^n$, $1\le t\le m$, that satisfy the following.
    \begin{itemize}
        \item {\textbf{Near-Optimality:}} For any $1\le i\le m$, $M^{-1}\mathcal{L}(\bs^{(i)})\le \kappa$, where $M=n\alpha_\gamma$, $\alpha_\gamma=\gamma 2^k\ln 2$.
        \item {\textbf{Overlap Constraint:}} For any $1\le t<\ell \le m$, $n^{-1}d_H(\bs^{(t)},\bs^{(\ell)})\in[\xi-\eta,\xi]$.
    \end{itemize}
\end{definition}
That is, $\mathcal{S}_{\mathrm{k-SAT}}(\gamma,m,\xi,\eta,\kappa)$ denotes the set of all nearly equidistant $m$-tuples of truth assignments such that for any $\bigl(\bs^{(1)},\dots,\bs^{(m)}\bigr)\in \mathcal{S}_{\mathrm{k-SAT}}(\gamma,m,\xi,\eta,\kappa)$ any $1\le i\le m$, fraction of clauses violated by $\bs^{(i)}$ is at most $\kappa$. In particular, when $\kappa=0$, we recover the original definition: $\mathcal{S}_{\mathrm{k-SAT}}(\gamma,m,\xi,\eta)=\mathcal{S}_{\mathrm{k-SAT}}(\gamma,m,\xi,\eta,0)$. 

We are now ready to state our final main result.
\begin{theorem}\label{thm:m-ogp-k-sat-absent-constant-k}
    For any $m\in \N$, $\gamma<1/m$, and $0<\eta<\xi<1$, there exists constants $C>0$, $K^*\in\N$ and $\zeta\in(0,1)$ such that the following holds. Fix any $k\ge K^*$. Then as $n\to\infty$,
    \[
\mathbb{P}\bigl[\mathcal{S}_{\mathrm{k-SAT}}\bigl(\gamma,m,\xi,\eta,C\zeta^{k/2}\bigr)\ne\varnothing\bigr]\ge 1-e^{-\Theta(n)}.
    \]
\end{theorem}
See below for the proof sketch and Section~\ref{sec:pf-thm:m-ogp-k-sat-absent-constant-k} for the complete proof. 

Theorem~\ref{thm:m-ogp-k-sat-absent-constant-k} asserts that for any $m\in\N$, $\gamma<1/m$ and $0<\eta<\xi<1$, there exists  $m$-tuples $(\bs^{(1)},\dots,\bs^{(m)})$ such that the normalized Hamming distances satisfy $n^{-1}d_H(\bs^{(i)},\bs^{(j)})\in[\xi-\eta,\xi]$, and for each $1\le i\le m$, $\bs^{(i)}$ satisfies at least $(1-O(\zeta^{k/2}))M$ many clauses. In other words, there exists $m$-tuples of nearly equidistant and nearly satisfying assignments below the density $(2^k\ln 2)/m$. 
\paragraph{Proof Sketch for Theorem~\ref{thm:m-ogp-k-sat-absent-constant-k}}
Recall from above that $\mathcal{S}_{\mathrm{k-SAT}}(\gamma,m,\xi,\eta):=\mathcal{S}_{\mathrm{k-SAT}}(\gamma,m,\xi,\eta,0)$. An inspection of the second moment calculation in the proof of Theorem~\ref{thm:m-ogp-k-sat-absent} reveals:
\begin{equation}\label{eq:2NDMOMKSAT}
    \mathbb{P}[\mathcal{S}_{\mathrm{k-SAT}}(\gamma,m,\xi,\eta)\ne\varnothing]\ge \exp(-2\alpha_\gamma n m^2 \bar{\zeta}^k),
\end{equation}
for some $\bar{\zeta}\in(0,1)$, where $\alpha_\gamma = \gamma 2^k\ln 2$. Next, we define the proxy random variable
\begin{equation}\label{eq:Z-beta-eta}
    Z_{\xi,\eta}:= \max_{\substack{\bs^{(1)},\dots,\bs^{(m)}\in\{0,1\}^n \\ n^{-1}d_H(\bs^{(i)},\bs^{(j)})\in[\xi-\eta,\xi]}}\, \min_{1\le i\le m} \widehat{\mathcal{L}}(\bs^{(i)}),
\end{equation}
where $\widehat{\mathcal{L}}(\bs) = M-\mathcal{L}(\bs)$ is the number of \emph{satisfied} clauses. Viewing $Z_{\xi,\eta}$ as a function of clauses $\mathcal{C}_1,\dots,\mathcal{C}_M$, we obtain 
$
|Z(\mathcal{C}_1,\dots,\mathcal{C}_M)-Z(\mathcal{C}_1,\dots,\mathcal{C}_{i-1},\mathcal{C}_i',\mathcal{C}_{i+1},\dots,\mathcal{C}_M)|\le 1$ for all $i$,
where $\mathcal{C}_i'$ is an independent copy of $\mathcal{C}_i$. So, $Z_{\xi,\eta}$ concentrates due to McDiarmid's inequality~\cite{warnke2016method}: 
\begin{equation}\label{eq:Z_CONNN}
    \mathbb{P}\bigl[\bigl|Z_{\xi,\eta}-\mathbb{E}[Z_{\xi,\eta}]\bigr|\ge t_0\bigr]\le 2\exp\bigl(-t_0^2 /M\bigr),
\end{equation}
for any $t_0\ge 0$.
Notice that $Z_{\xi,\eta}=M$ iff $|\mathcal{S}_{\mathrm{k-SAT}}(\gamma,m,\xi,\eta)|\ge 1$. Using~\eqref{eq:2NDMOMKSAT} and~\eqref{eq:Z_CONNN}, we have
\begin{align*}
    \mathbb{P}[Z_{\xi,\eta}=M] 
    &\ge \exp(-2\alpha_\gamma n m^2 \bar{\zeta}^k)\ge \exp(-(t^*)^2/M)\ge \mathbb{P}[Z_{\xi,\eta}\ge \mathbb{E}[Z_{\xi,\eta}]+t^*],
\end{align*}
where $t^* = \sqrt{3}n\alpha_\gamma m\bar{\zeta}^{k/2}$. Consequently, $\mathbb{E}[Z_{\xi,\eta}]\ge M-t^*=n\alpha_\gamma - t^*$. Applying~\eqref{eq:Z_CONNN} with $t_0=\sqrt{\alpha_\gamma}n$, we obtain that w.h.p., $Z_{\xi,\eta}\ge \mathbb{E}[Z_{\xi,\eta}]-\sqrt{\alpha_\gamma}n\ge n\alpha_\gamma(1-C\zeta^{k/2})$, where $C>0$ and $\zeta\in(0,1)$ are suitable constants. Finally, recalling
\[
Z_{\gamma,\eta}\ge n\alpha_\gamma(1-C\zeta^{k/2}) \iff  |\mathcal{S}_{\mathrm{k-SAT}}(\gamma,m,\xi,\eta,C\zeta^{k/2})|\ge 1
\]
we complete the proof of Theorem~\ref{thm:m-ogp-k-sat-absent-constant-k}.
\section{Future Directions \& Further Background on OGP}\label{sec:future-work}
An immediate direction is to complement Theorem~\ref{thm:m-ogp-k-sat-absent-constant-k} with an $m$-OGP result and to establish a sharp threshold for `almost solutions' to random $k$-SAT, analogous to Theorems~\ref{thm:m-ogp-sharp-PT} and \ref{thm:m-ogp-sharp-PT-sat}. Extensions to other spin glass models, such as mixed $p$-spin or multi-species models, are also
intriguing directions for future work.
\paragraph{Random $k$-SAT for Constant $k$} Extending Theorem~\ref{thm:m-ogp-k-sat-absent} to constant $k$ remains an intriguing question we leave for future work. A promising starting point is the NAE-$k$-SAT model, a variant of $k$-SAT where each clause must contain at least one true and one false literal. The inherent symmetry of NAE-$k$-SAT makes it particularly amenable to rigorous analysis. While the proof of satisfiability conjecture for random $k$-SAT was a long and arduous journey, culminating in the breakthrough tour de force by Ding, Sly, and Sun~\cite{ding2015proof}, analogous results for the NAE-$k$-SAT were established much earlier, using fairly standard second moment calculations by Achlioptas and Moore~\cite{achlioptas2002asymptotic}. This suggests that sharper results may be attainable for the NAE-$k$-SAT, potentially avoiding $k=\Omega(\log n)$ growth.
\paragraph{OGP and Algorithms} The interplay between the OGP and algorithms remains a key frontier, with far-reaching implications for average-case computational complexity. While the OGP formally implies the failure of large classes of algorithms, certain tractable optimization problems---solvable by linear programming---still exhibit the OGP~\cite{li2024some}. What makes these cases exceptional? What is the broadest class of algorithms that the OGP can provably rule out?

Likewise, what are the implications of the absence of the OGP? For instance, the algorithm by Montanari~\cite{montanari2019FOCS}, finding a near-ground state for the SK model, crucially relies on the widely-believed conjecture that the model does not exhibit OGP. Our results indicate that the absence of OGP at a certain level $m$ does not necessarily imply algorithmic tractability. Could there exist models without OGP that still pose insurmountable computational barriers? Finally, our findings reveal an intriguing scaling phenomenon: for certain models, the power of $m$-OGP in establishing algorithmic hardness strengthens indefinitely as $m$ increases. When does this trend persist and when does it saturate?
\vspace{-0.1in}

% \crefalias{section}{appendix} % uncomment if you are using cleveref
\subsubsection*{Further Background on the Overlap Gap Property}\label{sec:OGP-background}
The Ising $p$-spin and the random $k$-SAT models are two prototypical examples exhibiting a  \emph{statistical-computational gap} (\texttt{SCG}), where the best known algorithm perform strictly worse than what is existentially achievable. While standard complexity theory often falls short of explaining such average-case hardness (with notable exceptions such as~\cite{ajtai1996generating,boix2021average,GK-SK-AAP}), several frameworks have emerged to characterize \texttt{SCG}s and provide rigorous evidence of computational hardness. For a broader perspective on these approaches, we refer the reader to excellent surveys~\cite{bandeira2018notes,gamarnik2021overlap,gamarnik2022disordered}. 
\paragraph{Overlap Gap Property (OGP)} As discussed earlier, prior work~\cite{mezard2005clustering,achlioptas2006solution,achlioptas2008algorithmic} discovered an intriguing link between the geometric properties of the solution space and algorithmic hardness in certain random constraint satisfaction problems: the onset of clustering roughly coincides with the threshold beyond which no efficient algorithm is known. While this connection was highly suggestive, these works did not provide formal algorithmic hardness results. The first rigorous hardness results leveraging geometric obstructions came from Gamarnik and Sudan~\cite{gamarnik2014limits} through the OGP framework. 
Since then, OGP has proven to be a powerful tool for establishing (nearly sharp) algorithmic lower bounds across a wide range of average-case models, including random graphs~\cite{gamarnik2014limits,gamarnik2017,rahman2017local,gamarnik2020low,wein2020optimal,gamarnik2025optimal}, random CSPs~\cite{gamarnik2017performance,bresler2021algorithmic}, spin glasses~\cite{chen2019suboptimality,gamarnik2020low,gamarnikjagannath2021overlap,huang2021tight,huang2023algorithmic}, binary perceptron~\cite{gamarnik2022algorithms,gamarnik2023geometric,li2024discrepancy}, and number balancing~\cite{gamarnik2021algorithmic} as well as its planted variant~\cite{kizildaug2023planted}. For a comprehensive overview on OGP, see the survey by Gamarnik~\cite{gamarnik2021overlap}.
\paragraph{Multi OGP ($m$-OGP)} The OGP framework was first applied to largest independent set problem in sparse random graphs with $n$ vertices and average degree $d$ by Gamarnik and Sudan~\cite{gamarnik2014limits}. The largest independent set has asymptotic size $2\frac{\log d}{d}n$~\cite{frieze1992independence,bayati2010combinatorial}, yet the best known algorithm achieves only $\frac{\log d}{d}n$. For this model, Gamarnik and Sudan~\cite{gamarnik2014limits,gamarnik2017} showed that independent sets larger than $(1+1/\sqrt{2})\frac{\log d}{d}n$ exhibit the OGP: they either overlap substantially or are nearly disjoint, leading to the failure of local algorithms. Rahman and Vir{\'a}g~\cite{rahman2017local} achieved a sharp lower bound matching $\frac{\log d}{d}n$ by extending OGP to many independent sets ($m$-OGP). Gamarnik and Sudan~\cite{gamarnik2017performance} introduced the symmetric $m$-OGP, asserting the absence of tuples of nearly equidistant solutions, obtaining nearly sharp bounds for the NAE-$k$-SAT. A similar approach also yielded sharp lower bounds for binary perceptron~\cite{gamarnik2022algorithms,gamarnik2023geometric,li2024discrepancy} and lower bounds well below the optimal value for random number partitioning problem~\cite{gamarnik2022random,gamarnik2021algorithmic}.
More recently, asymmetric variants of the OGP have emerged, incorporating intricate overlap patterns where the $i\mathrm{th}$ solution has intermediate overlap with the first $i-1$ solutions. These structures have led to sharp lower bounds against low-degree polynomials~\cite{wein2020optimal,bresler2021algorithmic}. Even more recently, Huang and Sellke~\cite{huang2021tight,huang2023algorithmic} introduced the branching OGP, establishing tight hardness guarantees for the $p$-spin model via an ultrametric tree of solutions.

\section{Proofs}\label{sec:PFs}
We begin by listing some additional notation used below. For any $r\in\R^+$, denote by logarithm and the exponential functions base $r$, respectively by $\log_r(\cdot)$ and $\exp_r(\cdot)$; when $r=e$, we denote the former by $\ln(\cdot)$ and the latter by $\exp(\cdot)$. For $q\in[0,1]$, let $h_b(q)$ denotes the entropy of a Bernoulli variable with parameter $q$, i.e.\,$h_b(q)= -q\log_2 q -(1-q)\log_2(1-q)$. 

\subsection{Auxiliary Results}
Below, we collect several auxiliary results. The first one is a simple counting estimate, see~\cite[Theorem~3.1]{galvin2014three} for a proof.
\begin{lemma}\label{lemma:COUN}
    Fix $\alpha\le 1/2$. Then, for all $n$,
    \[
    \sum_{i\le \alpha n}\binom{n}{i}\le 2^{nh_b(\alpha)}.
    \]
\end{lemma}
We employ \emph{second moment method} via Paley-Zygmund inequality, recorded below.
\begin{lemma}\label{lemma:PZ}
    Let $Z$ be a random variable with $\mathbb{P}[Z\ge 0]=1$ and $\mathrm{Var}(Z)<\infty$. Then, for any $\theta\in[0,1]$,
    \[
    \mathbb{P}[Z>\theta\mathbb{E}[Z]]\ge (1-\theta)^2\frac{\mathbb{E}[Z]^2}{\mathbb{E}[Z^2]}.
    \]
\end{lemma}
For a proof, see~\cite{alon2016probabilistic}. The next result is a tail bound for multivariate normal random vectors. It is originally due to Savage~\cite{savage1962mills}; the version below is reproduced verbatim from~\cite{hashorva2003multivariate,hashorva2005asymptotics}. 
\begin{theorem}\label{thm:multiv-tail}
Let $\boldsymbol{X}\in\R^d$ be a centered multivariate normal random vector with non-singular covariance matrix $\Sigma\in\R^{d\times d}$ and $\boldsymbol{t}\in\R^d$ be a fixed threshold. Suppose that $\Sigma^{-1}\boldsymbol{t}>\boldsymbol{0}$ entrywise. Then,
\[
1-\ip{1/(\Sigma^{-1}\boldsymbol{t})}{\Sigma^{-1}(1/(\Sigma^{-1}\boldsymbol{t})} \le \frac{\mathbb{P}[\boldsymbol{X}\ge \boldsymbol{t}]}{\varphi_{\boldsymbol{X}}(\boldsymbol{t})\prod_{i\le d}\ip{e_i}{\Sigma^{-1}\boldsymbol{t}}}\le 1,
\]
where $e_i\in\R^d$ is the $i\mathrm{th}$ unit vector and $\varphi_{\boldsymbol{X}}(\boldsymbol{t})$ is the multivariate normal density evaluated at $\boldsymbol{t}$:
\[
\varphi_{\boldsymbol{X}}(\boldsymbol{t}) = (2\pi)^{-d/2}|\Sigma|^{-1/2}\exp\left(-\frac{\boldsymbol{t}^T\Sigma^{-1}\boldsymbol{t}}{2}\right)\in\R^+.
\]
\end{theorem}
\noindent We employ Theorem~\ref{thm:multiv-tail} when the coordinates of $X$ are nearly independent, i.e.\,$\Sigma$ is close to identity. The next auxiliary result we record is Slepian's Lemma~\cite{slepian1962one}.
\begin{lemma}\label{lemma:Slep}
        Let $X=(X_1,\dots,X_n)\in\R^n$ and $Y=(Y_1,\dots,Y_n)\in\R^n$ be two multivariate normal random vectors with (a) $\mathbb{E}[X_i]=\mathbb{E}[Y_i]=0$ and $\mathbb{E}[X_i^2]=\mathbb{E}[Y_i^2]$ for every $1\le i\le n$ and (b) $\mathbb{E}[X_iX_j]\le\mathbb{E}[Y_iY_j]$ for $1\le i<j\le m$. Fix any $c_1,\dots,c_n\in\R$. Then,
        \[
        \mathbb{P}\bigl[X_i\ge c_i,\forall i\bigr]\le 
        \mathbb{P}\bigl[Y_i\ge c_i,\forall i\bigr]
        \]
    \end{lemma}
The next auxiliary result regards the random $k$-SAT model.
\begin{lemma}\label{lemma:random-k-sat-prob}
    Let $\bs,\bs'\in\{0,1\}^n$ with $d_H(\bs,\bs')\ge \xi n$ for some $\xi\in(0,1)$. Suppose the clause $\mathcal{C}$ is a disjunction of $k$ literals $y_1,\dots,y_k$, that is  $\mathcal{C} = y_1\vee y_2 \vee \cdots \vee y_k$,     where $y_i$ are drawn independently and uniformly at random from $\{x_1,\dots,x_n,\bar{x}_1,\dots,\bar{x}_n\}$. Then,
    \[
    \mathbb{P}\bigl[\mathcal{C}(\bs)=\mathcal{C}(\bs') = 0\bigr] \le 2^{-k}(1-\xi)^k.
    \]
\end{lemma}
\begin{proof}[Proof of Lemma~\ref{lemma:random-k-sat-prob}]
    Fix $\bs,\bs'\in\{0,1\}^n$ with $d_H(\bs,\bs')=\xi n$, let $I=\{i\in[n]:\bs(i)\ne \bs'(i)\}$, so that $|I|\ge \xi n$. Note that if there is an $i$ such that $x_i\in \{y_1,\dots,y_k\}$ or $\bar{x}_i\in\{y_1,\dots,y_k\}$, then at least one of $\mathcal{C}(\bs)$ or $\mathcal{C}(\bs')$ is satisfied. Consequently, $\mathcal{C}(\bs)=\mathcal{C}(\bs')=0$ iff $\mathcal{C}$ consists of literals $x_i$/$\bar{x}_i$ for which $i\in [n]\setminus I$. Denote this event by $\mathcal{E}_g$. So,
    \begin{align*}
         \mathbb{P}\bigl[\mathcal{C}(\bs)=\mathcal{C}(\bs') = 0\bigr]  &=  \mathbb{P}\bigl[\mathcal{C}(\bs)=\mathcal{C}(\bs') = 0\mid \mathcal{E}_g\bigr] \mathbb{P}[\mathcal{E}_g] \\
         &=2^{-k}\cdot \frac{(n-d_H(\bs,\bs'))^k}{n^k} \\
         &\le 2^{-k}(1-\xi)^k. 
        \end{align*}
\end{proof}
    %\noindent The last auxiliary result is Wielandt-Hoffman inequality~\cite{hoffman1953variation} (see also~\cite[Corollary~6.3.8]{horn2012matrix}).
  %  \begin{theorem}\label{thm:wiehoff}
     %   Let $A,A+E\in\R^{n\times n}$ be two symmetric matrices with eigenvalues
      %  \[
      %  \lambda_1(A)\ge \cdots\ge \lambda_n(A)\quad\text{and}\quad  \lambda_1(A+E)\ge \cdots\ge \lambda_n(A+E).
      %  \]
      %  Then,
      %  \[
      %  \sum_{1\le i\le n}\left(\lambda_i(A+E)-\lambda_i(A)\right)^2 \le \|E\|_F^2.
      %  \]
   % \end{theorem}
%\end{proof}

\subsection{Proof of Theorem~\ref{thm:m-ogp-tight}}\label{sec:pf-m-ogp-tight}
We first observe that if $0<\eta'<\eta<\xi$ then $
\mathcal{S}_{\mathrm{p-spin}}(\gamma,m,\xi,\eta')\subseteq \mathcal{S}_{\mathrm{p-spin}}(\gamma,m,\xi,\eta)$. In particular, $\mathcal{S}_{\mathrm{p-spin}}(\gamma,m,\xi,\eta')\ne\varnothing\Rightarrow \mathcal{S}_{\mathrm{p-spin}}(\gamma,m,\xi,\eta)\ne\varnothing$. For this reason, we assume throughout that $\eta$ is sufficiently small.
\paragraph{Existence of $m$-Tuples with Constrained Overlaps} For any $m\in\N$, $0<\eta<\xi<1$, let
\begin{equation}\label{eq:overlap-set}
    \mathcal{F}(m,\xi,\eta):= \left\{\left(\bs^{(1)},\dots,\bs^{(m)}\right):\xi-\eta\le n^{-1}\ip{\bs^{(i)}}{\bs^{(j)}}\le \xi,1\le i<j\le m\right\}.
    \end{equation}
We first show that $\mathcal{F}(m,\xi,\eta)\ne\varnothing$ for all large enough $n$.
\begin{lemma}\label{lemma:F-NON-EMPTY}
    For any $m\in\N$, $0<\eta<\xi<1$, there exists $N_0\in \N$ such that $\mathcal{F}(m,\xi,\eta)\ne\varnothing$ for $n\ge N_0$. 
\end{lemma}
\begin{proof}[Proof of Lemma~\ref{lemma:F-NON-EMPTY}]
    Our approach is through the \emph{probabilistic method}~\cite{alon2016probabilistic}. Choose $\delta>0$ such that $\xi-\eta<\xi-2\delta<\xi$ and $\xi-2\delta>0$, and let $p^*\in(0,1)$ satisfy
    \[
    (1-2p^*)^2 = \xi - \delta.
    \]
    As $\xi-\delta\in(0,1)$ such a $p^*$ indeed exists. We now assign coordinates of $\bs^{(i)}$ randomly. Specifically, let $\bs^{(i)}(k)$ be i.i.d.\,for $1\le i\le m$ and $1\le k\le n$ with distribution
    \[
    \mathbb{P}\bigl[\bs^{(i)}(k) = 1\bigr] = p^*\quad\text{and}\quad \mathbb{P}\bigl[\bs^{(i)}(k) = -1\bigr]=1-p^*.
    \]
    Observe that for fixed $1\le i<j\le m$, random variables $\bs^{(i)}(k)\bs^{(j)}(k),1\le k\le n$ are i.i.d.\,with
    \[
    \mathbb{E}[\bs^{(i)}(k)\bs^{(j)}(k)] = (p^*)^2 + (1-p^*)^2 -2p^*(1-p^*) = (1-2p^*)^2 = \xi -\delta.
    \]
    In particular, $\ip{\bs^{(i)}}{\bs^{(j)}}$ is a sum of i.i.d.\,binary random variables with $n^{-1}\mathbb{E}[\ip{\bs^{(i)}}{\bs^{(j)}}] = \xi-\delta$.  Define next sequence of events
    \[
    \mathcal{E}_{ij} = \{n^{-1}\ip{\bs^{(i)}}{\bs^{(j)}}\in [\xi-2\delta,\xi]\},\quad 1\le i<j\le m.
    \]
    Using standard concentration results for sum of i.i.d.\,sub-Gaussian random variables (see e.g.\,\cite{vershynin2018high}) we obtain that there is a $C>0$ such that for any $\delta>0$ and $n$, 
    \[
\mathbb{P}\left[\left|\frac1n\sum_{1\le k\le n}\bs^{(i)}(k)\bs^{(j)}(k)-(\xi-\delta)\right|\le \delta\right]\ge 1-\exp\bigl(-Cn\delta^2\bigr).
    \]
So, 
\[
\mathbb{P}\left[\bigcap_{1\le i,j\le m}\mathcal{E}_{ij}\right]\ge 1-\binom{m}{2}e^{-Cn\delta^2}>0,
\]
for $n$ large enough. This establishes the existence of $\bs^{(i)}\in\Sigma_n,1\le i\le m$ with $n^{-1}\ip{\bs^{(i)}}{\bs^{(j)}}\in [\xi-2\delta,\xi]\subset[\xi-\delta,\xi]$ for $1\le i<j\le m$, and completes the proof.
\end{proof}
Equipped with Lemma~\ref{lemma:F-NON-EMPTY}, fix a $\bs^*\in\Sigma_n$ and let 
\begin{equation}\label{eq:L_term}
    L :=L(m,\xi,\eta) = \bigl|\left\{(\bs^{(i)}:i\le m):\bs^{(1)}=\bs^*,n^{-1}\ip{\bs^{(i)}}{\bs^{(j)}}\in[\xi-\eta,\xi],1\le i<j\le m\right\}\bigr|.
\end{equation}
From symmetry, (a) $L$ is also the number of $m$-tuples $(\bs^{(i)}:i\le m)\in\mathcal{F}(m,\xi,\eta)$ with $k\mathrm{th}$ coordinate fixed at $\bs^*$, $\bs^{(k)}=\bs^*$ for any arbitrary $k$, and (b) $|\mathcal{F}(m,\xi,\eta)| =2^n\cdot L$, so $L\ge 1$ for large $n$ by Lemma~\ref{lemma:F-NON-EMPTY}.

\paragraph{An Auxiliary Random Variable} Next, define 
\begin{equation}\label{eq:auxil-rv-T}
    T_{m,\xi,\eta}:= \max_{(\bs^{(i)}:i\le m)\in\mathcal{F}(m,\xi,\eta)} \min_{1\le j\le m}H(\bs^{(j)}).
\end{equation}
We establish a concentration property for $T_{m,\eta,\xi}$.
\begin{lemma}\label{lemma:T-concen}
For any $t\ge 0$, 
\[
\mathbb{P}\bigl[\bigl|T_{m,\xi,\eta}-\mathbb{E}[T_{m,\xi,\eta}]\bigr|\ge t\bigr] \le 2\exp\left(-\frac{t^2 n}{2}\right).
\]
\end{lemma}
\begin{proof}[Proof of Lemma~\ref{lemma:T-concen}]
Throughout the proof, we make dependence on disorder $\boldsymbol{J}$ explicit. Given $\boldsymbol{J}\in(\R^n)^{\otimes p}$, we view $T_{m,\xi,\eta}$ as a map $T_{m,\xi,\eta}(\boldsymbol{J}):\R^{n^p}\to \R$ and establish that it is Lipschitz (w.r.t.\,$\ell_2$ norm):
    \begin{equation}\label{eq:T-is-Lip}
        \bigl|T_{m,\xi,\eta}(\boldsymbol{J}) -T_{m,\xi,\eta}(\boldsymbol{J}') \bigr|\le n^{-\frac12}\|\boldsymbol{J}-\boldsymbol{J}'\|_2.
    \end{equation}
    Lemma~\ref{lemma:T-concen} then follows from standard concentration results for Lipschitz functions of Gaussian random variables (see e.g.~\cite[Theorem~2.26]{wainwright2019high}). To that end, fix any $\bs$ and recall that
    \[
    H(\bs,\boldsymbol{J}) = n^{-\frac{p+1}{2}}\ip{\boldsymbol{J}}{\bs^{\otimes p}}. 
    \]
    Using Cauchy-Schwarz inequality and the fact $\|\bs^{\otimes p}\|_2=n^{p/2}$ for any $\bs\in\Sigma_n$, we obtain
    \begin{equation}\label{eq:ham-lip}
           \bigl|H(\bs,\boldsymbol{J}) - H(\bs,\boldsymbol{J}')\bigr| \le n^{-\frac{p+1}{2}}\|\boldsymbol{J}-\boldsymbol{J}'\|_2 \|\bs^{\otimes p}\|_2 = n^{-\frac12}\|\boldsymbol{J}-\boldsymbol{J}'\|_2.
        \end{equation}
Now, fix any $\Xi = (\bs^{(i)}:i\le m)\in\mathcal{F}(m,\xi,\eta)$ and let $L(\Xi,\boldsymbol{J}) = \min_{1\le j\le m}H(\bs^{(j)},\boldsymbol{J})$. We have 
\[
L(\Xi,\boldsymbol{J}) \le H(\bs^{(j)},\boldsymbol{J}) \le H(\bs^{(j)},\boldsymbol{J}') +n^{-\frac12}\|\boldsymbol{J}-\boldsymbol{J}'\|_2,
\]
where we used~\eqref{eq:ham-lip} in the last step. Taking a minimum over $1\le j\le m$ and exchanging $\boldsymbol{J}$ with $\boldsymbol{J}'$ afterwards, we obtain 
\begin{equation}\label{eq:L-lip}
   \bigl| L(\Xi,\boldsymbol{J}) - L(\Xi,\boldsymbol{J}')\bigr| \le n^{-\frac12}\|\boldsymbol{J}-\boldsymbol{J}'\|_2.
\end{equation}
Using~\eqref{eq:L-lip}, we have
\begin{align*}
    L(\Xi,\boldsymbol{J})&\le L(\Xi,\boldsymbol{J}')+n^{-\frac12}\|\boldsymbol{J}-\boldsymbol{J}'\|_2 \\
    &\le \max_{\Xi\in\mathcal{F}(m,\xi,\eta)}L(\Xi,\boldsymbol{J}')+n^{-\frac12}\|\boldsymbol{J}-\boldsymbol{J}'\|_2\\
    &= T_{m,\xi,\eta}(\boldsymbol{J}') + n^{-\frac12}\|\boldsymbol{J}-\boldsymbol{J}'\|_2.
\end{align*}
Taking a maximum over $\Xi\in\mathcal{F}(m,\xi,\eta)$ and reversing the roles of $\boldsymbol{J}$ and $\boldsymbol{J}'$, we establish~\eqref{eq:T-is-Lip}. 
\end{proof}
\paragraph{Bounds on Moments of $|\mathcal{S}_{\mathrm{p-spin}}(\gamma,m,\xi,\eta)|$} We establish a lower bound on first moment of $|\mathcal{S}_{\mathrm{p-spin}}(\gamma,m,\xi,\eta)|$.
\begin{proposition}
    \label{prop:S_1st_Mom}
 For any $m\in\N$,  $\gamma<1/\sqrt{m}$ and $0<\eta<\xi<1$,
 \[
\mathbb{E}\bigl[|\mathcal{S}_{\mathrm{p-spin}}(\gamma,m,\xi,\eta)|\bigr]\ge L\cdot \exp_2\bigl(n-m\gamma^2 n +o(n)\bigr),
\]
for $L$ in~\eqref{eq:L_term}. In particular, $\mathbb{E}\bigl[|\mathcal{S}_{\mathrm{p-spin}}(\gamma,m,\xi,\eta)|\bigr]=\exp(\Theta(n))$ for $\gamma<1/\sqrt{m}$.
\end{proposition}
\begin{proof}[Proof of Proposition~\ref{prop:S_1st_Mom}]
    We have 
\begin{equation}\label{eq:1ST_M_Exp}
        |\mathcal{S}_{\mathrm{p-spin}}(\gamma,m,\xi,\eta)| = \sum_{(\bs^{(i)}:i\le m)\in\mathcal{F}(m,\xi,\eta)}\ind\left\{H(\bs^{(i)})\ge \gamma\sqrt{2\ln 2},\forall i\right\}.
        \end{equation}
  %  We begin by recalling Slepian's lemma~\cite{slepian1962one}.
    Now, fix any $(\bs^{(i)}:i\le m)\in\mathcal{F}(m,\xi,\eta)$ and let $Z_i=\sqrt{n}H(\bs^{(i)}),1\le i\le n$. We have $Z_i\sim \cN(0,1)$ for $1\le i\le m$. Furthermore, a simple calculation shows
    \[
    \mathbb{E}[Z_iZ_j] = \left(n^{-1}\ip{\bs^{(i)}}{\bs^{(j)}}\right)^p \in [(\xi-\eta)^p,\xi^p].
    \]
    Let now $Z_i'\sim \cN(0,1)$ be i.i.d. Using the fact $\xi-\eta>0$, we apply Slepian's Lemma, Lemma~\ref{lemma:Slep}, to random vectors $(Z_1,\dots,Z_n)$ and $(Z_1',\dots,Z_n')$ to obtain
\begin{equation}\label{eq:1st--mom-p-bd}
\mathbb{P}\bigl[H(\bs^{(i)})\ge \gamma\sqrt{2\ln 2},\forall i\bigr] \ge \mathbb{P}\bigl[Z_i'\ge \gamma\sqrt{2n \ln 2},\forall i\bigr]=(\mathbb{P}[Z_1'\ge \gamma\sqrt{2n \ln 2}])^m \ge \exp_2\bigl(-nm\gamma^2+o(n)\bigr),
    \end{equation}
    where we used the well-known Gaussian tail bound
    \[
    \mathbb{P}[\cN(0,1)\ge x]\ge \frac{\exp(-x^2/2)}{\sqrt{2\pi}}\left(\frac1x-\frac{1}{x^3}\right)
    \]
    with $x=\gamma\sqrt{2n\ln 2}$. We now combine~\eqref{eq:1ST_M_Exp},~\eqref{eq:1st--mom-p-bd} and the fact $|\mathcal{F}(m,\xi,\eta)| = 2^n\cdot L$ per discussion below~\eqref{eq:L_term}, and establish Proposition~\ref{prop:S_1st_Mom} via linearity of expectation.
\end{proof}

We next establish the following proposition using the \emph{second moment method}.
\begin{proposition}\label{prop:2nd-moment}
    For any $m\in\N$, $0<\gamma<1/\sqrt{m}$ and $0<\eta<\xi<1$, there exists a $P^*\in \N$ and a function $\Delta_p:=\Delta_p(m,\xi,\eta)>0$, depending only on $p,m,\gamma,\xi$ with the property that $\Delta_p\to 0$ as $p\to\infty$ (for fixed $m,\gamma,\xi$), such that the following holds. Fix any $p\ge P^*$. Then, as $n\to\infty$,
    \[
\mathbb{P}\bigl[\bigl|\mathcal{S}_{\mathrm{p-spin}}(\gamma,m,\xi,\eta)\bigr|\ge 1\bigr]\ge \exp\left(-\frac{2mn\gamma^2 \Delta_p}{1+\Delta_p}+o(n)\right).
    \]
\end{proposition}
\begin{proof}[Proof of Proposition~\ref{prop:2nd-moment}]
Fix $m\in\N$, $\gamma<1/\sqrt{m}$, $0<\eta<\xi<1$, where $\eta$ is as small as needed (see the remark at the beginning of the proof). For convenience, let
\begin{equation}\label{eq:COUNT-RV}
    M := |\mathcal{S}_{\mathrm{p-spin}}(\gamma,m,\xi,\eta)| = \sum_{\Xi \in\mathcal{F}(m,\xi,\eta)} \ind\bigl\{H(\bs^{(i)})\ge \gamma\sqrt{2\ln 2}, 1\le i\le m\bigr\},
\end{equation}
where 
\begin{equation}\label{eq:XI_Term}
    \Xi = (\bs^{(i)}:1\le i\le m)\in\mathcal{F}(m,\xi,\eta).
\end{equation}
Observe that for $\bar{\Xi}=(\bar{\bs}^{(i)}:1\le i\le m)\in\mathcal{F}(m,\xi,\eta)$, we have
\begin{equation}\label{eq:2nd_Mom_Exp}
    \mathbb{E}[M^2] = \sum_{\Xi,\bar{\Xi}\in\mathcal{F}(m,\xi,\eta)}\mathbb{P}\bigl[H(\bs^{(i)})\ge \gamma\sqrt{2\ln 2},H(\bar{\bs}^{(i)})\ge \gamma\sqrt{2\ln 2}, 1\le i\le m\bigr]
\end{equation}
Since $\gamma<1/\sqrt{m}$, there exists an $\epsilon^*\in(0,\frac12)$ such that
\begin{equation}\label{eq:Eps-star}
    h_b(\epsilon^*)<1-m\gamma^2.
\end{equation}
Our goal is to carefully upper bound~\eqref{eq:2nd_Mom_Exp} using a crucial overcounting idea developed in~\cite{gamarnik2021algorithmic}.
\paragraph{Overcounting} For $\epsilon^*$ in~\eqref{eq:Eps-star}, let
\begin{equation}\label{eq:Sigma-ij}
    \Sigma_{ij}(\epsilon^*):= \sum_{\substack{\Xi,\bar{\Xi}\in\mathcal{F}(m,\xi,\eta) \\ n^{-1}d_H(\bs^{(i)},\bar{\bs}^{(j)})\in[0,\epsilon^*]\cup[1-\epsilon^*,1]}} \mathbb{P}\bigl[H(\bs^{(t)})\ge \gamma\sqrt{2\ln 2},H(\bar{\bs}^{(t)})\ge \gamma\sqrt{2\ln 2}, 1\le t\le m\bigr].
\end{equation}
and
\begin{equation}\label{eq:Sigma_D}
    \Sigma_d :=  \sum_{\substack{\Xi,\bar{\Xi}\in\mathcal{F}(m,\xi,\eta) \\ n^{-1}d_H(\bs^{(i)},\bar{\bs}^{(j)})\in[\epsilon^*,1-\epsilon^*],\forall i,j}} \mathbb{P}\bigl[H(\bs^{(t)})\ge \gamma\sqrt{2\ln 2},H(\bar{\bs}^{(t)})\ge \gamma\sqrt{2\ln 2}, 1\le t\le m\bigr].
\end{equation}
Observe the bound following from overcounting $m$-tuples:
\begin{equation}\label{eq:overcounting}
    \mathbb{E}[M^2] \le \sum_{1\le i,j\le m}\Sigma_{ij}(\epsilon^*) + \Sigma_d.
\end{equation}
\paragraph{The $\Sigma_{ij}(\epsilon^*)$ Term} We now show $\Sigma_{ij}(\epsilon^*)/\mathbb{E}[M]^2 = \exp(-\Theta(n))$. To that end, we establish the following.
\begin{lemma}\label{lemma:prob_term}
Suppose $m\in\N$ and $0<\eta<\xi<1$ such that $\eta$ is small enough. Then, for $\Xi$ in~\eqref{eq:XI_Term},
    \[
    \sup_{\Xi\in\mathcal{F}(m,\xi,\eta)} \mathbb{P}\bigl[H(\bs^{(i)})\ge \gamma\sqrt{2\ln 2},1\le i\le m\bigr]\le \exp_2\left(-n\frac{m\gamma^2}{1+2mp\xi^p}+O(\log_2 n)\right).
    \]
\end{lemma}
\begin{proof}[Proof of Lemma~\ref{lemma:prob_term}]
Fix
\[
\Xi :=(\bs^{(i)}:1\le i\le m)\in\mathcal{F}(m,\xi,\eta)\quad\text{with}\quad \frac1n \ip{\bs^{(k)}}{\bs^{(\ell)}} = \xi-\eta_{k\ell},
\]
where $1\le k<\ell \le m$. Denote by $\Sigma(\boldsymbol{\eta})$ the covariance matrix of $(\sqrt{n}H(\bs^{(i)}):1\le i\le m)\in\R^m$, where $\boldsymbol{\eta} = (\eta_{k\ell}:1\le k<\ell \le m)$.
\begin{lemma}{\cite[Lemma~3.13(d)]{gamarnik2023shattering}}\label{lemma:From-GJK}
    \begin{align*}
    & \mathbb{P}\left[\min_{1\le i\le m}H(\bs^{(i)})\ge \gamma\sqrt{2\ln 2}\right]\\
    &\le n^{\frac{m}{2}}\left(\gamma\sqrt{\frac{\ln 2}{\pi}}\right)^m\left(\prod_{1\le i\le m}\ip{e_i}{\Sigma(\boldsymbol{\eta})^{-1}\boldsymbol{1}}\right)\bigl|\Sigma(\boldsymbol{\eta})\bigr|^{-\frac12}\exp_2\left(-\gamma^2 n\boldsymbol{1}^T \cdot \Sigma(\boldsymbol{\eta})^{-1}\boldsymbol{1}\right).
        \end{align*}
\end{lemma}
To control the terms appearing in Lemma~\ref{lemma:From-GJK}, we again appeal to~\cite{gamarnik2023shattering}. First,~\cite[Equation~(93)] {gamarnik2023shattering} establishes
\begin{equation}\label{eq:93}   
|\Sigma(\boldsymbol{\eta})|^{-\frac12} \le \bigl(1-2mp\xi^p\bigr)^{-\frac{m}{2}} = O_n(1),
\end{equation}
since $m,p,\xi=O_n(1)$. Next,~\cite[Equation~(94)] {gamarnik2023shattering} gives
\begin{equation}\label{eq:94}    
\prod_{i\le m}\ip{e_i}{\Sigma(\boldsymbol{\eta})^{-1}\boldsymbol{1}}\le m^{\frac{m}{2}}\bigl(1-2mp\xi^p\bigr)^{-m} = O_n(1).
\end{equation}
Lastly,~\cite[Equation~(95)] {gamarnik2023shattering} gives
\begin{equation}\label{eq:95}   
\boldsymbol{1}^T\Sigma(\boldsymbol{\eta})^{-1}\boldsymbol{1}\ge \frac{m}{1+2mp\xi^p}.
\end{equation}
Combining Lemma~\ref{lemma:From-GJK} with~\eqref{eq:93}-\eqref{eq:95} and observing that the bounds~\eqref{eq:93}-\eqref{eq:95} do not depend on $\boldsymbol{\eta}$, we establish Lemma~\ref{lemma:prob_term}.
%follows by combining~\cite[Lemma~3.13(d)]{gamarnik2023shattering} and~\cite[Equations~(93),(94),(95)]{gamarnik2023shattering} through a similar reasoning leading to~\cite[Equation~(97)]{gamarnik2023shattering}.
\end{proof}
Equipped with Lemma~\ref{lemma:From-GJK}, define
\[
S_{ij}(\epsilon^*) := \bigl\{(\Xi,\bar{\Xi})\in\mathcal{F}(m,\xi,\eta)\times \mathcal{F}(m,\xi,\eta):n^{-1}d_H(\bs^{(i)},\bar{\bs}^{(j)})\in[0,\epsilon^*]\cup[1-\epsilon^*,1]\bigr\}.
\]
Using Lemma~\ref{lemma:COUN}, and the remark following~\eqref{eq:L_term}, we obtain
\begin{equation}\label{eq:S_ij_eps-star}
    |S_{ij}(\epsilon^*)| \le \bigl(2^n \cdot L\bigr)\cdot \left(\sum_{k\in [0,n\epsilon^*]\cup [n(1-\epsilon^*),n]}\binom{n}{k}\cdot L\right) \le \exp_2\bigl(n+1+nh_b(\epsilon^*)\bigr)L^2.
\end{equation}
Furthermore, for any $(\Xi,\bar{\Xi})\in\mathcal{F}(m,\xi,\eta)\times \mathcal{F}(m,\xi,\eta)$, we have
\begin{align}
\mathbb{P}\bigl[H(\bs^{(i)})\ge \gamma\sqrt{2\ln 2},H(\bar{\bs}^{(i)})\ge \gamma\sqrt{2\ln 2}, 1\le i\le m\bigr]&\le \mathbb{P}\bigl[H(\bs^{(i)})\ge \gamma\sqrt{2\ln 2}, 1\le i\le m\bigr]\nonumber \\
&\le \exp_2\left(-n\frac{m\gamma^2}{1+2mp\gamma^p}+O(\log_2 n)\right)\label{eq:Sigma-ij-eps-star-prob},
\end{align}
where we used Lemma~\ref{lemma:prob_term} to deduce~\eqref{eq:Sigma-ij-eps-star-prob}. Combining~\eqref{eq:S_ij_eps-star} and~\eqref{eq:Sigma-ij-eps-star-prob}, we upper bound $\Sigma_{ij}(\epsilon^*)$ in~\eqref{eq:Sigma-ij}:
\begin{equation}\label{eq:UPPER__BD_S_IJ}
    \Sigma_{ij}(\epsilon^*) \le L^2 \exp_2\left(n+1+nh_b(\epsilon^*) -n\frac{m\gamma^2}{1+2mp\xi^p}+O(\log_2 n)\right)
\end{equation}
Furthermore, Proposition~\ref{prop:S_1st_Mom} yields
\begin{equation}\label{eq:1ST_MOM_Sq}
    \mathbb{E}[M]^2 \ge L^2 \exp_2\bigl(2n-2m\gamma^2 n+o(n)\bigr).
\end{equation}
Now, recall $\epsilon^*$ from~\eqref{eq:Eps-star} and choose $P_1^*$ such that 
\begin{equation}\label{eq:P_1-STERKA}
    1-h_b(\epsilon^*) - m\gamma^2 - \frac{2m^2\gamma^2 p\xi^p}{1+2mp\xi^p}>0,\quad\text{for all}\quad p\ge P_1^*.
\end{equation}
Such a $P_1^*$ indeed exists since $1-h_b(\epsilon^*)-m\gamma^2>0$ and $2m^2\gamma^2 p\xi^p/(1+2mp\xi^p)\to 0$ as $p\to\infty$ for fixed $m,\gamma$ and $\xi\in(0,1)$. Equipped with these, we now combine~\eqref{eq:UPPER__BD_S_IJ} and~\eqref{eq:1ST_MOM_Sq} and obtain
\begin{align}
\frac{\Sigma_{ij}(\epsilon^*)}{\mathbb{E}[M]^2} &\le \exp_2\left(-n+nh_b(\epsilon^*) +2m\gamma^2 n - \frac{m\gamma^2 n}{1+2mp\xi^p}+o(n)\right) \nonumber \\
&=\exp_2\left(-n\left(1-h_b(\epsilon^*)-m\gamma^2-\frac{2m^2\gamma^2 p\xi^p}{1+2mp\xi^p}\right)+o(n)\right)\nonumber \\
&=\exp_2\bigl(-\Theta(n)\bigr)\label{eq:S-ij-eps-over-exp-small}
\end{align}
for any $p\ge P_1^*$ via~\eqref{eq:P_1-STERKA}. Since $m=O(1)$ and the indices $1\le i,j\le m$ are arbitrary, we thus obtain
\begin{equation}\label{eq:WEAK_TERM}
   \frac{1}{\mathbb{E}[M]^2}\sum_{1\le i,j\le m}\Sigma_{ij}(\epsilon^*) \le 2^{-\Theta(n)}, \quad \text{for}\quad p\ge P_1^*.
\end{equation}
\paragraph{The $\Sigma_d$ Term} We now study $\Sigma_d$ term in~\eqref{eq:Sigma_D}. Let
\begin{equation}\label{eq:S_d_eps-star}
    S_d(\epsilon^*) := \bigl\{(\Xi,\bar{\Xi})\in\mathcal{F}(m,\xi,\eta)\times \mathcal{F}(m,\xi,\eta):n^{-1}d_H(\bs^{(i)},\bar{\bs}^{(j)})\in[\epsilon^*,1-\epsilon^*],1\le i,j\le m\bigr\}.
\end{equation}
We crudely upper bound $|S_d(\epsilon^*)|$ by
\begin{equation}\label{eq:COUNT-Sigma_D}
\bigl|S_d(\epsilon^*)\bigr|\le |\mathcal{F}(m,\xi,\eta)|^2.
\end{equation}
We next establish the key probability estimate.
\begin{lemma}\label{lemma:key-prob}
    For $m\in\N$, $0<\eta<\xi<1$, and $\epsilon^*$ in~\eqref{eq:Eps-star}, there is a $P_2^*\in \N$ such that the following holds. Fix any $p\ge P_2^*$. Then,
    \[
    \sup_{(\Xi,\bar{\Xi})\in S_d(\epsilon^*)}\mathbb{P}\left[\min_{1\le i\le m}H(\bs^{(i)})\ge \gamma\sqrt{2\ln 2},\min_{1\le i\le m}H(\bar{\bs}^{(i)})\ge \gamma\sqrt{2\ln 2}\right]\le \exp_2\left(-\frac{2m\gamma^2 n}{1+\Delta_p}+O(\log_2 n)\right),
    \]
    where
    \[
    \Delta_p = m\sqrt{2\xi^{2p}+2(1-2\epsilon^*)^{2p}}.
    \]
\end{lemma}
\begin{proof}[Proof of Lemma~\ref{lemma:key-prob}]
    Fix $(\Xi,\bar{\Xi})\in S_d(\epsilon^*)$ with $\Xi=(\bs^{(i)}:i\le m)$ and $\bar{\Xi}=(\bar{\bs}^{(i)}:i\le m)$. From~\eqref{eq:S_d_eps-star},
\begin{equation}\label{eq:Overlap_intra}
        \frac1n\ip{\bs^{(i)}}{\bar{\bs}^{(j)}} \in[-(1-2\epsilon^*),1-2\epsilon^*],\quad 1\le i,j\le m.
        \end{equation}
        Moreover, as $\Xi,\bar{\Xi}\in\mathcal{F}(m,\xi,\eta)$, we also have for $1\le i<j\le m$ \begin{equation}\label{eq:Overlap_inter}
            \frac1n\ip{\bs^{(i)}}{\bs^{(j)}}\in[\xi-\eta,\xi]\quad\text{and}\quad \frac1n\ip{\bar{\bs}^{(i)}}{\bar{\bs}^{(j)}}\in[\xi-\eta,\xi].
        \end{equation}
        Set $Z_i = \sqrt{n}H(\bs^{(i)}),1\le i\le m$ and $\bar{Z}_i=\sqrt{n}H(\bar{\bs}^{(i)}),1\le i\le m$, so that $Z_i,\bar{Z}_i\sim \cN(0,1)$. We proceed by applying Theorem~\ref{thm:multiv-tail} to the random vector $\boldsymbol{V}\in\R^{2m}$ with tail $\boldsymbol{t}\in\R^{2m}$, where
    \begin{equation}\label{eq:V_and_t}
                \boldsymbol{V}=(Z_1,\dots,Z_m,\bar{Z}_1,\dots,\bar{Z}_m)\in\R^{2m}\quad\text{and}\quad\boldsymbol{t} =(\gamma\sqrt{2n\ln 2})\mathbf{1}\in\R^{2m}.
               \end{equation}
To that end, we study the covariance matrix $A(\Xi,\bar{\Xi})\in\R^{2m\times 2m}$ of $\boldsymbol{V}$, admitting the following form:
                \begin{equation}\label{eq:A-Xi-bar-Xi}
        A(\Xi,\bar{\Xi}) =  \displaystyle\begin{pmatrix}
B(\Xi)
  & \vline & \hat{E} \\
\hline
  \hat{E}& \vline &
 B(\bar{\Xi})
\end{pmatrix}\in\R^{2m\times 2m}.      
        \end{equation}
    The matrices $B(\Xi),B(\bar{\Xi}),\hat{E}\in\R^{m\times m}$ satisfy the following properties.
    \begin{itemize}
        \item For $1\le i<j\le m$,
        \[
       \bigl(B(\Xi)\bigr)_{ij} =  \bigl(B(\Xi)\bigr)_{ji} = \left(n^{-1}\ip{\bs^{(i)}}{\bs^{(j)}}\right)^p \in [(\xi-\eta)^p,\xi^p]
        \]
        and
        \[
             \bigl(B(\bar{\Xi})\bigr)_{ij} =  \bigl(B(\bar{\Xi})\bigr)_{ji} = \left(n^{-1}\ip{\bar{\bs}^{(i)}}{\bar{\bs}^{(j)}}\right)^p \in [(\xi-\eta)^p,\xi^p]
        \]
        Moreover, for $1\le i\le m$,
        \[
        \bigl(B(\Xi)\bigr)_{ii} = \bigl(B(\bar{\Xi})\bigr)_{ii} = 1.
        \]
        \item For $1\le i,j\le m$, 
        \[
        \hat{E}_{ij} =\hat{E}_{ji}=\left(n^{-1}\ip{\bs^{(i)}}{\bar{\bs}^{(j)}}\right)^p \Rightarrow |\hat{E}_{ij}| = |\hat{E}_{ji}| \le (1-2\epsilon^*)^p.
        \]
    \end{itemize}
    Equipped with these, we write $A(\Xi,\bar{\Xi})$ in~\eqref{eq:A-Xi-bar-Xi} as
    \begin{equation}\label{eq:Id-decompose}
        A(\Xi,\bar{\Xi}) = I_{2m} + E\in\R^{2m\times 2m},
    \end{equation}
    where
    \begin{equation}\label{eq:E-Frob_norm}
        \|E\|_F \le m\sqrt{2\xi^{2p}+2(1-2\epsilon^*)^{2p}} := \Delta_p.
    \end{equation}
    Note that as $m=O(1)$,  $\xi\in(0,1)$ and $\epsilon^*\in(0,\frac12)$, $\Delta_p\to 0$ as $p\to\infty$. Before applying Theorem~\ref{thm:multiv-tail}, we must verify $\bigl(A(\Xi,\bar{\Xi})^{-1}\boldsymbol{t}\bigr)_i>0$ for $1\le i\le 2m$. We have the following chain of inequalities:
    \begin{align}
\bigl\|A(\Xi,\bar{\Xi})^{-1}\boldsymbol{t} -\boldsymbol{t} \bigr\|_2 &\le \|\boldsymbol{t}\|_2\cdot \bigl\|\bigl(I_{2m}+E\bigr)^{-1}-I_{2m}\bigr\|_2 \label{eq:CSE}\\
&\le \gamma\sqrt{4mn\ln 2}\cdot \left\|\sum_{k\ge 1}E^k\right\|_F\label{eq:matrix-Taylor} \\
&\le \gamma\sqrt{4mn\ln 2}\left(\sum_{k\ge 1}\|E\|_F^k\right)\label{eq:T_ineq} \\
&\le \gamma\sqrt{4mn\ln 2}\frac{\Delta_p}{1-\Delta_p}\label{eq:Geo_Series}.
    \end{align}
    Here,~\eqref{eq:CSE} follows from Cauchy-Schwarz inequality;~\eqref{eq:matrix-Taylor} follows from the fact $\boldsymbol{t}=(\gamma\sqrt{2n\ln 2})\mathbf{1}\in\R^{2m}$ and the series $I+E=I-E-E^2-\cdots$ valid for $\|E\|_2<1$;~\eqref{eq:T_ineq} follows from triangle inequality; and~\eqref{eq:Geo_Series} follows from~\eqref{eq:E-Frob_norm} and the geometric series summation. 

    Using~\eqref{eq:Geo_Series}, we have
    \[
\bigl|\bigl(A(\Xi,\bar{\Xi})^{-1}\boldsymbol{t}\bigr)_i - \gamma\sqrt{2n \ln 2}\bigr| \le \bigl\|A(\Xi,\bar{\Xi})^{-1}\boldsymbol{t} - \boldsymbol{t}\bigr\|_2 \le \gamma\sqrt{4mn\ln 2}\frac{\Delta_p}{1-\Delta_p}.
    \]
    In particular, there is a $P_2^*\in\N$ such that for $p\ge P_2^*$, we have
    \[
    \bigl(A(\Xi,\bar{\Xi})^{-1}\boldsymbol{t}\bigr)_i\ge \gamma\sqrt{2n\ln 2}\left(1-\frac{\Delta_p\sqrt{2m}}{1-\Delta_p}\right)>0,\quad 1\le i\le m,
    \]
    where we used the fact $\Delta_p\to 0$ as $p\to 0$. In particular, Theorem~\ref{thm:multiv-tail} is applicable, provided $p\ge P_2^*$. Below, the value of the constant $P_2^*\in\N$ changes from line to line.
    \paragraph{Controlling the Spectrum of $A(\Xi,\bar{\Xi})$} Using~\eqref{eq:Id-decompose}, we control the spectrum of $A(\Xi,\bar{\Xi})$. Recall Weyl's Inequality~\cite{horn2012matrix}: if $A,B$ are Hermitian matrices then 
    \[
    \bigl|\lambda_i(A+B)-\lambda_i(A)\bigr|\le \|B\|_2 \le \|B\|_F,
    \]
    where $\lambda_1(A)\ge \lambda_2(A)\ge \cdots$ are the eigenvalues of $A$. Let $\mu_1\ge\cdots\ge \mu_{2m}$ be the eigenvalues of $A(\Xi,\bar{\Xi})=I_{2m}+E$ per~\eqref{eq:Id-decompose}. Since the eigenvalues of $I_{2m}$ are all unity, Weyl's inequality yields:
    \begin{equation}\label{eq:Spec_of_A}
       \bigl|\mu_i-1\bigr| \le \|E\|_F \Rightarrow \mu_i\in\bigl[1-\Delta_p,1+\Delta_p\bigr],\quad\text{for}\quad 1\le i\le 2m.
        %\sum_{1\le i\le 2m}\bigl(\mu_i-1\bigr)^2 \le \|E\|_F^2 \Rightarrow \mu_i\in\bigl[1-\Delta_p,1+\Delta_p\bigr],\quad\text{for}\quad 1\le i\le 2m.
    \end{equation}
    Note that for $p$ large enough, $\mu_{2m}>0$, so $A(\Xi,\bar{\Xi})$ is positive definite. Using~\eqref{eq:Spec_of_A} we obtain
    \begin{equation}\label{eq:A-det-bd}
        |A(\Xi,\bar{\Xi})| = \prod_{1\le i\le 2m}\mu_i \ge (1-\Delta_p)^{2m}.
    \end{equation}
    Next, we control certain quantities before applying Theorem~\ref{thm:multiv-tail}. Diagonalize $A(\Xi,\bar{\Xi})^{-1}$ as 
    \[
    A(\Xi,\bar{\Xi}) = Q(\Xi,\bar{\Xi})^T \Lambda(\Xi,\bar{\Xi}) Q(\Xi,\bar{\Xi}),
    \]
    where 
    \begin{equation}\label{eq:Q_normal}
        Q(\Xi,\bar{\Xi})\in\R^{2m\times 2m}\quad\text{with}\quad  Q(\Xi,\bar{\Xi})^T  Q(\Xi,\bar{\Xi}) =  Q(\Xi,\bar{\Xi}) Q(\Xi,\bar{\Xi})^T =I_{2m}
        \end{equation}
    and
    \[
    \Lambda(\Xi,\bar{\Xi}) = \mathrm{diag}\bigl(\mu_i^{-1}:1\le i\le 2m\bigr)\in \R^{2m\times 2m}.
    \]
    Set $\boldsymbol{v}_{\mathrm{aux}}=Q(\Xi,\bar{\Xi})\mathbf{1}\in\R^{2m}$. Using~\eqref{eq:Q_normal}, we obtain $\|\boldsymbol{v}_{\mathrm{aux}}\|_2^2=2m$. With this, we have
    \begin{align*}
            \boldsymbol{t}^T A(\Xi,\bar{\Xi})^{-1}\boldsymbol{t}
            &=(\gamma^2 \cdot 2n\ln 2)(\boldsymbol{v}_{\mathrm{aux}})^T \Lambda(\Xi,\bar{\Xi})\boldsymbol{v}_{\mathrm{aux}} =(\gamma^2 \cdot 2n\ln 2) \sum_{1\le i\le 2m}\mu_i^{-1}\bigl(\boldsymbol{v}_{\mathrm{aux}}\bigr)_i^2 \ge \frac{4mn\gamma^2 \ln 2}{1+\Delta_p},
        \end{align*}
        where we used the fact $\mu_i^{-1}\ge 1+\Delta_p$ per~\eqref{eq:Spec_of_A} and $\|\boldsymbol{v}_{\mathrm{aux}}\|_2^2 = 2m$ in the last step. So,
        \begin{equation}\label{eq:EXP_TERM}
         \exp\left(-\frac{\boldsymbol{t}^TA(\Xi,\bar{\Xi})^{-1}\boldsymbol{t}}{2}\right)  \le \exp_2\left(-\frac{2mn\gamma^2}{1+\Delta_p}\right).
        \end{equation}
    Lastly, fix any $1\le i\le 2m$ and let $e_i$ be the $i\mathrm{th}$ unit vector in $\R^{2m}$. We have,
    \begin{align}
        \bigl|\ip{e_i}{A(\Xi,\bar{\Xi})^{-1}\boldsymbol{t}}\bigr|&\le \|e_i\|_2 \cdot \|A(\Xi,\bar{\Xi})^{-1}\|_2 \cdot \|\boldsymbol{t}\|_2\label{eq:CSE-2} \\
        &\le \frac{2\gamma\sqrt{mn \ln 2}}{1-\Delta_p}\label{eq:combine_norm_spec},
    \end{align}
    where~\eqref{eq:CSE-2} follows by applying Cauchy-Schwarz inequality twice and~\eqref{eq:combine_norm_spec} follows by combining the facts $\|e_i\|_2=1$ and $\|\boldsymbol{t}\|_2 = 2\gamma\sqrt{mn\ln 2}$ with~\eqref{eq:Spec_of_A}. Using the fact $A(\Xi,\bar{\Xi})^{-1}\boldsymbol{t}$ is entrywise positive for large $p$ as established earlier, we obtain
    \begin{equation}\label{eq:pi_of_proj}
        \prod_{1\le i\le 2m} \ip{e_i}{A(\Xi,\bar{\Xi})^{-1}\boldsymbol{t}} \le n^m\cdot \left(\frac{2\gamma\sqrt{m\ln 2}}{1-\Delta_p}\right)^{2m}.
    \end{equation}
     \paragraph{Applying Theorem~\ref{thm:multiv-tail}} We apply  Theorem~\ref{thm:multiv-tail} to $\boldsymbol{V}\in\R^{2m}$ in~\eqref{eq:V_and_t} with $\boldsymbol{t}\in\R^{2m}$ to obtain
     \begin{align}
         &\mathbb{P}\left[\min_{1\le i\le m}H(\bs^{(i)})\ge \gamma\sqrt{2\ln 2},\min_{1\le i\le m}H(\bar{\bs}^{(i)})\ge \gamma\sqrt{2\ln 2}\right] \nonumber\\
         &\le (2\pi)^{-m} |A(\Xi,\bar{\Xi})|^{-1} \exp\left(-\frac{\boldsymbol{t}^TA(\Xi,\bar{\Xi})^{-1}\boldsymbol{t}}{2}\right)  \prod_{1\le i\le 2m} \ip{e_i}{A(\Xi,\bar{\Xi})^{-1}\boldsymbol{t}} \nonumber \\
         &=\exp_2\left(-\frac{2mn\gamma^2}{1+\Delta_p} +O(\log_2 n)\right)\label{eq:prob_combine_everything}
     \end{align}
where~\eqref{eq:prob_combine_everything} follows by combining~\eqref{eq:A-det-bd},~\eqref{eq:EXP_TERM} and~\eqref{eq:pi_of_proj} and recalling that as $m,\Delta_p=O(1)$, we have
\[
(1-\Delta_p)^{2m}, \left(\frac{2\gamma\sqrt{m\ln 2}}{1-\Delta_p}\right)^{2m} =O(1)\quad\text{and}\quad \log_2(n^m) = O(m\log_2 n)=O(\log_2 n).
\]
Observe that both $O(\log_2 n)$ term as well as $\Delta_p$ are independent of $(\Xi,\bar{\Xi})$. Taking a supremum over all such pairs, we deduce
\[
\sup_{(\Xi,\bar{\Xi})\in S_d(\epsilon^*)}\mathbb{P}\left[\min_{1\le i\le m}H(\bs^{(i)})\ge \gamma\sqrt{2\ln 2},\min_{1\le i\le m}H(\bar{\bs}^{(i)})\ge \gamma\sqrt{2\ln 2}\right]\le \exp_2\left(-\frac{2m\gamma^2 n}{1+\Delta_p}+O(\log_2 n)\right)
\]
and establish Lemma~\ref{lemma:key-prob}.
     \end{proof}
We combine Lemma~\ref{lemma:key-prob} with the counting estimate~\eqref{eq:COUNT-Sigma_D} to obtain an upper bound on $\Sigma_d$ in~\eqref{eq:Sigma_D}:
\begin{equation}\label{eq:Sigma_d_up}
    \Sigma_d \le |\mathcal{F}(m,\xi,\eta)|^2 \exp_2\left(-\frac{2mn\gamma^2}{1+\Delta_p}+O(\log_2 n)\right).
\end{equation}
Using \eqref{eq:Sigma_d_up} and~\eqref{eq:1ST_MOM_Sq} and recalling $|\mathcal{F}(m,\xi,\eta)|=2^n\cdot L$ per~\eqref{eq:L_term} we obtain that for $p\ge P_2^*$,
\begin{equation}\label{eq:STRONG_TERM}
    \frac{\Sigma_d}{\mathbb{E}[M]^2}\le \exp_2\left(\frac{2mn\gamma^2 \Delta_p}{1+\Delta_p}+o(n)\right).
\end{equation}
We now apply Paley-Zygmund inequality, Lemma~\ref{lemma:PZ}, to random variable $M$ with $\theta=0$: for $p$ large,
\[
\mathbb{P}\bigl[M\ge 1\bigr] \ge \frac{\mathbb{E}[M]^2}{\mathbb{E}[M^2]}\ge \frac{1}{\exp\bigl(-\Theta(n)\bigr) + \exp_2\left(\frac{2mn\gamma^2 \Delta_p}{1+\Delta_p}+o(n)\right)}\ge \exp_2\left(-\frac{2mn\gamma^2 \Delta_p}{1+\Delta_p}+o(n)\right),
\]
where we combined~\eqref{eq:overcounting},~\eqref{eq:WEAK_TERM}, and~\eqref{eq:STRONG_TERM} in the penultimate step. Lastly, note that $\Delta_p$ is defined in~\eqref{eq:E-Frob_norm}; it depends only on $p,m,\xi$ and $\epsilon^*$. As $\epsilon^*$ per~\eqref{eq:Eps-star} is a function of $m,\gamma$ only, we find that $\Delta_p$ is indeed a function of $p,m,\gamma$, and $\xi$ only. With this, we finish the proof of Proposition~\ref{prop:2nd-moment}.
\end{proof}
\paragraph{Putting Everything Together} Equipped with Lemma~\ref{lemma:T-concen} and Proposition~\ref{prop:2nd-moment}, we now show how to establish Theorem~\ref{thm:m-ogp-tight}. Fix a $\gamma'$ such that $\gamma<\gamma'<1/\sqrt{m}$. Using Proposition~\ref{prop:2nd-moment} with $\gamma'$, we obtain that there is a $c\in(0,1)$ and a $P^*\in\N$ such that for any fixed $p\ge P^*$, 
\begin{equation}\label{eq:2nd-mom-helps}
\mathbb{P}\bigl[\bigl|\mathcal{S}_{\mathrm{p-spin}}(\gamma',m,\xi,\eta)\bigr|\ge 1\bigr]\ge \exp\bigl(-nc^p\bigr),
\end{equation}
as $n\to\infty$. Observe that for $T_{m,\xi,\eta}$ in~\eqref{eq:auxil-rv-T},
\begin{equation}\label{eq:crucial}
    \bigl\{\bigl|\mathcal{S}_{\mathrm{p-spin}}(\gamma',m,\xi,\eta)\bigr|\ge 1\bigr\}= \bigl\{T_{m,\xi,\eta}\ge \gamma'\sqrt{2\ln 2}\bigr\}. 
\end{equation}
Next, note that for any fixed $\epsilon>0$, there is a $P_\epsilon\in \N$ such that for every $p\ge P_\epsilon$
\[
\frac{\epsilon^2}{2}>\frac{2m\gamma^2 \Delta_p}{1+\Delta_p}
\]
since $\Delta_p\to 0$ as $p\to\infty$. In the remainder, fix any $p\ge \max\{P_\epsilon,P^*\}$. We then have
\[
\exp\left(-\frac{2mn\gamma^2 \Delta_p}{1+\Delta_p}+o(n)\right)\ge 2\exp\left(-\frac{n\epsilon^2}{2}\right)
\]
for all large enough $n$. 

Combining Lemma~\ref{lemma:T-concen}, Proposition~\ref{prop:2nd-moment}, and the observation~\eqref{eq:crucial}, we obtain
\begin{equation}\label{eq:CHAIN}
    \mathbb{P}\bigl[T_{m,\xi,\eta}\ge \gamma'\sqrt{2\ln 2}\bigr] \ge \exp\left(-\frac{2mn\gamma^2 \Delta_p}{1+\Delta_p}+o(n)\right)\ge 2\exp\left(-\frac{n\epsilon^2}{2}\right) \ge \mathbb{P}\bigl[T_{m,\xi,\eta}\ge \mathbb{E}[T_{m,\xi,\eta}]+\epsilon\bigr].
\end{equation}
In particular, 
\begin{equation}\label{eq:E-LB}
\mathbb{E}[T_{m,\xi,\eta}]\ge \gamma'\sqrt{2\ln 2}-\epsilon.
\end{equation}
Lastly, we combine Lemma~\ref{lemma:T-concen} (with $t=\epsilon$) and~\eqref{eq:E-LB} to obtain that w.p.\,at least $1-\exp(-\Theta(n))$, \[
T_{m,\xi,\eta}\ge \gamma'\sqrt{2\ln 2}-2\epsilon.
\]
Provided $\epsilon<(\gamma'-\gamma)\sqrt{\ln 2/2}$, we obtain that \[
T_{m,\xi,\eta}\ge \gamma\sqrt{2\ln 2}
\]
w.p.\,at least $1-\exp(-\Theta(n))$. This together with~\eqref{eq:crucial}, establishes Theorem~\ref{thm:m-ogp-tight}.
    \subsection{Proofs of Theorem~\ref{thm:m-ogp-sharp-PT}}\label{pf:PT}
  %  \begin{proof}[Proof of Theorem~\ref{thm:m-ogp-sharp-PT}]
    Fix $m\in\N$. Using Proposition~\ref{thm:m-ogp}, we obtain that any $\gamma>1/\sqrt{m}$ is $m$-inadmissible in the sense of Definition~\ref{def:adm-exp}. In particular,
    \[
    \inf\{\gamma>0:\gamma\text{ is $m$-inadmissible}\} = 1/\sqrt{m}.
    \]
    Similarly, Theorem~\ref{thm:m-ogp-tight} shows that any $\gamma<1/\sqrt{m}$ is $m$-admissible. Taken together, we find
    \[
    \sup\{\gamma>0:\gamma\text{ is $m$-admissible}\} = 1/\sqrt{m}.
    \]
    Recalling Definition~\ref{def:sharp-PT}, we establish Theorem~\ref{thm:m-ogp-sharp-PT} by combining the above two displays.

\subsection{Proof of Theorem~\ref{thm:m-ogp-k-sat}}\label{sec:pf-ogp-k-sat}
Our proof is based on the \emph{first moment method}. Fix $m\in\N$, $0<\eta<\xi<\frac12$ and let 
\begin{equation}\label{eq:S-beta}
\widehat{\mathcal{F}}(m,\xi,\eta)=\bigl\{(\bs^{(1)},\dots,\bs^{(m)}):n^{-1}d_H(\bs^{(k)},\bs^{(\ell)})\in[\xi-\eta,\xi],1\le k<\ell\le m\bigr\},
\end{equation}
where $\bs^{(t)}\in\{0,1\}^n$. Fix a $\gamma>1/m$. Define 
\begin{equation}\label{eq:rv-N}
N = \bigl|\mathcal{S}_{\mathrm{k-SAT}}(\gamma,m,\xi,\eta)\bigr| = \sum_{(\bs^{(t)}:t\in[m])\in\widehat{\mathcal{F}}(m,\xi,\eta)} \ind\bigl\{\Phi(\bs^{(t)}) = 1,\forall t\in[m]\bigr\}.
\end{equation}
In what follows, we establish that $\mathbb{E}[N]=e^{-\Theta(n)}$ for a suitable choice of $\xi$ and $\eta$. The conclusion then follows by Markov's inequality: $\mathbb{P}[N\ge 1]\le \mathbb{E}[N]$. 
\paragraph{Cardinality Estimate} We first bound $|\widehat{\mathcal{F}}(m,\xi,\eta)|$. Note that, there are $2^n$ choices for $\bs^{(1)}$. Having fixed a $\bs^{(1)}$, the number of choices for any $\bs^{(k)}$, $2\le k\le m$, is $\sum_{(\xi-\eta)n\le k\le \xi n}\binom{n}{k}$. Provided $\xi<\frac12$, which will be satisfied by our eventual choice of parameters,  we have
\[
\sum_{(\xi-\eta)n\le k\le \xi n}\binom{n}{k}\le \binom{n}{\xi n}n,
\]
using the fact binomial coefficients decrease away from the
center. So, 
\begin{align}\label{eq:card-S}
|\widehat{\mathcal{F}}(m,\xi,\eta)| 
&\le 2^n \left(\binom{n}{\xi n}n^{O(1)}\right)^{m-1}\le \exp\Bigl(n\ln 2 + n(m-1)\ln 2 \cdot h_b(\xi) + O(\ln n)\Bigr),
\end{align}
where we used Stirling's approximation, $\binom{n}{\xi n} = \exp_2\bigl(nh_b(\xi)+O(\ln n)\bigr)$. 
\paragraph{Probability Estimate} We next study the probability term. Fix a $(\bs^{(1)},\dots,\bs^{(t)})\in\widehat{\mathcal{F}}(m,\xi,\eta)$. Let $\Phi$ consists of independent $k$-clauses $\mathcal{C}_i$, $1\le i\le M$, where $M=\alpha_\gamma n$ for $\alpha_\gamma = \gamma 2^k\ln 2$. We have
\begin{align}
\mathbb{P}\bigl[\Phi(\bs^{(t)})=1,\forall t\in[m]\bigr] &=\mathbb{P}\bigl[\mathcal{C}_1(\bs^{(t)})=1,\forall t\in[m]\bigr]^{\alpha_\gamma n} \nonumber\\
    &=\Bigl(1-\mathbb{P}\bigl[\cup_{1\le t\le m}\{\mathcal{C}_1(\bs^{(t)})=0\}\bigr]\Bigr)^{\alpha_\gamma n}\nonumber\\
    &\le \Bigl(1-m\cdot 2^{-k}+\binom{m}{2}(1-\xi+\eta)^k 2^{-k}\Bigr)^{\alpha_\gamma n}\label{eq:prob-term-ksat},
\end{align}
where~\eqref{eq:prob-term-ksat} is obtained by using the fact that for events $\mathcal{E}_1,\dots,\mathcal{E}_m$,
\[
\mathbb{P}\bigl[\cup_{1\le i\le m}\mathcal{E}_i\bigr]\ge \sum_{1\le i\le m}\mathbb{P}[\mathcal{E}_i] - \sum_{1\le i<j\le m}\mathbb{P}\bigl[\mathcal{E}_i\cap\mathcal{E}_j\bigr],
\]
and Lemma~\ref{lemma:random-k-sat-prob}.
\paragraph{Combining Everything} We combine~\eqref{eq:card-S} and~\eqref{eq:prob-term-ksat} to arrive at
\begin{align*}
    \mathbb{E}[N]\le \exp\left(n\ln 2 + n(m-1)\ln 2\cdot h_b(\xi) +\alpha_\gamma n \ln\left(1-m\cdot 2^{-k}+\binom{m}{2}(1-\xi+\eta)^k 2^{-k}\right)+O(\ln n)\right).
\end{align*}
Next, a Taylor expansion yields $\ln(1-x)=-x+O(x^2)$. Thus,
\[
\ln\left(1-m\cdot 2^{-k} + \binom{m}{2}(1-\xi+\eta)^k 2^{-k}\right) = -m\cdot 2^{-k} + \binom{m}{2}(1-\xi+\eta)^k 2^{-k} +O_k\left(m^2 2^{-2k}\right),
\]
using the fact that $m=O_k(1)$. We now insert $\alpha_\gamma =\gamma 2^k\ln 2$ (for a $\gamma>1/m$) to obtain
\begin{align*}
    \mathbb{E}[N]&\le \exp\left(n\ln 2\underbrace{\left(1+(m-1)h_b(\xi)-m\gamma+\gamma\binom{m}{2}(1-\xi+\eta)^k +O_k\left(m^2 2^{-k}\right) \right)}_{:=\varphi(k,\gamma,m,\xi,\eta)}+O(\ln n)\right)\\
    &=\exp\Bigl(n\cdot \ln 2 \cdot \varphi(k,\gamma,m,\xi,\eta)+O(\ln n)\Bigr).
\end{align*}
Since $\gamma > \frac1m$, we have that $\gamma m = 1+\delta$ for some $\delta>0$. With this, choose $\xi^*<1/2$, such that $(m-1)h_b(\xi^*)\le \delta/4$. Set $\eta^* = \xi^*/2$. Finally, choose $K^*$ sufficiently large, so that, for $k\ge K^*$ (and aforementioned choice of $m,\gamma,\xi^*$),
\[
\gamma \binom{m}{2} (1-\xi^*+\eta^*)^k + O_k\left(m^2 2^{-k}\right) \le \frac{\delta}{4}.
\]
With these, we obtain
\[
\varphi(k,\gamma,m,\xi^*,\eta^*)\le -\frac{\delta}{2},\quad \forall k\ge K^*
\]
and therefore
\[
\mathbb{E}[N]\le \exp\left(-\frac{\delta\ln 2}{2}n+O(\ln n)\right) = \exp\bigl(-\Theta(n)\bigr).
\]
This completes the proof of Theorem~\ref{thm:m-ogp-k-sat}.
\subsection{Proof of Theorem~\ref{thm:m-ogp-k-sat-absent}}\label{sec:pf-ogp-absent-k-sat}
The proof is based on the \emph{second moment method}. To that end, recall Paley-Zygmund inequality from Lemma~\ref{lemma:PZ}, which asserts that for any non-negative integer-valued random variable $N$,
\begin{equation}\label{eq:paley}
    \mathbb{P}\bigl[N\ge 1\bigr]\ge \mathbb{E}[N]^2/\mathbb{E}[N^2].
\end{equation}
%Indeed, observing $\mathbb{P}[N\ge 1]=\mathbb{E}[\ind\{N\ge 1\}]$, Cauchy-Schwarz inequality asserts that
%\[
%\mathbb{E}[N^2]\mathbb{P}[N\ge 1]\ge \mathbb{E}\bigl[N\ind\{N\ge 1\}\bigr]^2 = \mathbb{E}[N]^2.
%\]
%Since $N\ge 0$ a.s., we have
%\[
%\mathbb{E}[N]=\mathbb{E}[N\ind\{N\ge 1\}]+\mathbb{E}[N\ind\{N=0\}] = \mathbb{E}[N\ind\{N\ge 1\}],
%\]
%thus~\eqref{eq:paley} follows.

 Now, recall $\widehat{\mathcal{F}}(m,\xi,\eta)$ from~\eqref{eq:S-beta} and let $N$ be the random variable introduced in~\eqref{eq:rv-N}. Fix a $\bs\in\{0,1\}^n$ and define
 \[
L_{\bs}=\Bigl|\bigl\{(\bs^{(1)},\dots,\bs^{(m)})\in \widehat{\mathcal{F}}(m,\xi,\eta):\bs^{(1)}=\bs\bigr\}\Bigr|.
\]
As $L_{\bs}$ is independent of $\bs$, we set $L:= L_{\bs}$ and obtain that $|\widehat{\mathcal{F}}(m,\xi,\eta)|=2^n\cdot L$. Next, fix a $\gamma<1/m$ and let $\alpha_\gamma = \gamma 2^k\ln 2$.
\paragraph{Lower Bounding $\mathbb{E}[N]$} We lower bound $\mathbb{E}[N]$. Observe that,
\begin{align*}
\mathbb{P}\bigl[\Phi(\bs^{(t)})=1,\forall t\in[m]\bigr] &=\mathbb{P}\bigl[\mathcal{C}_1(\bs^{(t)})=1,\forall t\in[m]\bigr]^{\alpha_\gamma n} \nonumber\\
    &=\Bigl(1-\mathbb{P}\bigl[\cup_{1\le t\le m}\{\mathcal{C}_1(\bs^{(t)})=0\}\bigr]\Bigr)^{\alpha_\gamma n}\nonumber\\
    &\ge \Bigl(1-m\cdot 2^{-k}\Bigr)^{\alpha_\gamma n},
\end{align*}
where we took a union bound over $1\le t\le m$ to arrive at the last line. Consequently, 
\begin{equation}\label{eq:lower-bd-N}
    \mathbb{E}[N] \ge 2^n\cdot L\left(1-m\cdot 2^{-k}\right)^{\alpha n}.
\end{equation}
\paragraph{Upper Bounding $\mathbb{E}[N^2]$} Recall that $\gamma<1/m$. Fix an $\epsilon>0$ such that \begin{equation}\label{eq:choice-of-eps}
-\delta:= -1 + h_b(\epsilon)+m\gamma<0,
\end{equation}
where $h_b(p)=-p\log_2 p -(1-p)\log_2(1-p)$ is the binary entropy function log base 2. Such an $\epsilon>0$ indeed exists as $\gamma<\frac1m$ and $h_b(\cdot)$ is continuous with $\lim_{p\to 0}h_b(p)=0$.

Next, define the family $S_{ij}(\epsilon)$, $1\le i,j\le m$ of sets
\begin{equation}\label{eq:S-ij-eps}
    S_{ij}(\epsilon) = \Bigl\{(\Xi,\bar{\Xi})\in \widehat{\mathcal{F}}(m,\xi,\eta)\times\widehat{\mathcal{F}}(m,\xi,\eta):d_H(\bs^{(i)},\overline{\bs}^{(j)})\le \epsilon n\Bigr\},
\end{equation}
where $\Xi,\bar{\Xi}\in\widehat{\mathcal{F}}(m,\xi,\eta)$ are such that
\begin{equation}\label{eq:bs-sigma}
\Xi=\bigl(\bs^{(i)}:1\le i\le m)\quad\text{and}\quad \bar{\Xi}=\bigl(\overline{\bs}^{(i)}:1\le i\le m),
\end{equation}
Furthermore, let
\begin{equation}\label{eq:F-eps}
\mathcal{F}(\epsilon) = \bigl(\widehat{\mathcal{F}}(m,\xi,\eta)\times\widehat{\mathcal{F}}(m,\xi,\eta)\bigr)\setminus \bigcup_{1\le i,j\le m}S_{ij}(\epsilon).
\end{equation}
Observe that for every $(\Xi,\bar{\Xi})\in \mathcal{F}(\epsilon)$ and every $1\le i,j\le m$,  
\[
d_H\bigl(\bs^{(i)},\overline{\bs}^{(j)}\bigr)>\epsilon n.
\]
We next upper bound $\mathbb{E}[N^2]$. Observe that
\begin{align}
    \mathbb{E}[N^2] &= \sum_{(\Xi,\bar{\Xi})\in \widehat{\mathcal{F}}(m,\xi,\eta)\times \widehat{\mathcal{F}}(m,\xi,\eta)} \mathbb{P}\Bigl[\Phi(\bs^{(i)})=\Phi(\overline{\bs}^{(i)})=1,\forall i\in[m]\Bigr] \nonumber\\
    &= \underbrace{\sum_{1\le i,j\le m}\sum_{(\Xi,\bar{\Xi})\in S_{ij}(\epsilon)} \mathbb{P}\Bigl[\Phi(\bs^{(i)})=\Phi(\overline{\bs}^{(i)})=1,\forall i\in[m]\Bigr]}_{:= A_\epsilon}  + \underbrace{\sum_{(\Xi,\bar{\Xi})\in \mathcal{F}(\epsilon)} \mathbb{P}\Bigl[\Phi(\bs^{(i)})=\Phi(\overline{\bs}^{(i)})=1,\forall i\in[m]\Bigr]}_{:= B_\epsilon}.\label{eq:A-eps-B-eps-term}
\end{align}
\subsubsection*{The Term $A_\epsilon/\mathbb{E}[N]^2$}
We next show that the dominant contribution to upper bound in~\eqref{eq:A-eps-B-eps-term} comes from $B_\epsilon$. That is, we claim 
\[
\frac{A_\epsilon}{\mathbb{E}[N]^2}=\exp(-\Theta(n)).
\]
To that end, we first upper bound
\begin{equation}\label{eq:2nd-mom-prob}
\max_{\Xi,\bar{\Xi} \in \widehat{\mathcal{F}}(m,\xi,\eta)}\mathbb{P}\Bigl[\Phi(\bs^{(i)})=\Phi(\overline{\bs}^{(i)})=1,\forall i\in[m]\Bigr],
\end{equation}
where $\Xi$ and $\overline{\Xi}$ are as in~\eqref{eq:bs-sigma}.
%To that end, fix $\boldsymbol{\sigma},\overline{\boldsymbol{\sigma}}\in\mathcal{S}_\xi$. 
\begin{lemma}\label{lemma:2nd-mom-pb}
\begin{align*}
\max_{\Xi,\bar{\Xi} \in \widehat{\mathcal{F}}(m,\xi,\eta)}\mathbb{P}\Bigl[\Phi(\bs^{(i)})=\Phi(\overline{\bs}^{(i)})=1,\forall i\in[m]\Bigr] \le \left(1-m\cdot 2^{-k}+2m(m-1)\cdot 2^{-k}\left(1-\frac{\xi-\eta}{2}\right)^k\right)^{\alpha n}.
\end{align*}
\end{lemma}
\begin{proof}[Proof of Lemma~\ref{lemma:2nd-mom-pb}]
Fix $\Xi,\bar{\Xi}\in\widehat{\mathcal{F}}(m,\xi,\eta)$. Define the sets
\begin{equation}\label{eq:index-set}
I_i:= \left\{j\in[m]:d_H\bigl(\bs^{(i)},\overline{\bs}^{(j)}\bigr)<\frac{(\xi-\eta) n}{2}\right\},\quad 1\le i\le m.
\end{equation}

We claim that $|I_i|\le 1$ for all $1\le i\le m$. Indeed, suppose that this is false for some $i$. That is, there exists $1\le j<k\le m$ such that $d_H\bigl(\bs^{(i)},\overline{\bs}^{(j)}\bigr),d_H\bigl(\bs^{(i)},\overline{\bs}^{(k)}\bigr)<\frac{(\xi-\eta) n}{2}$. Using triangle inequality,
\[
(\xi-\eta) n \le  d_H\bigl(\overline{\bs}^{(j)},\overline{\bs}^{(k)}\bigr)\le d_H\bigl(\bs^{(i)},\overline{\bs}^{(j)}\bigr)+d_H\bigl(\bs^{(i)},\overline{\bs}^{(k)}\bigr)<(\xi-\eta) n,
\]
a contradiction. Thus, $|I_i|\le 1, \forall i\in[m]$.

For $1\le i\le m$, introduce the events $\mathcal{E}_i = \bigl\{\mathcal{C}_1(\bs^{(i)})=0\bigr\}$, $\overline{\mathcal{E}}_i =\bigl\{\mathcal{C}_1(\overline{\bs}^{(i)})=0\bigr\}$. We have,
\begin{equation}\label{eq:2-to-k}
\mathbb{P}\bigl[\mathcal{E}_i\bigr] = \mathbb{P}\bigl[\overline{\mathcal{E}}_i\bigr] = 2^{-k},\quad 1\le i\le m.
\end{equation}
Next, we have the following chain:
\begin{align}
&\mathbb{P}\Bigl[\Phi(\bs^{(i)})=\Phi(\overline{\bs}^{(i)})=1,\forall i\in[m]\Bigr]  \nonumber \\
&=\mathbb{P}\Bigl[\mathcal{C}_1(\bs^{(i)})=\mathcal{C}_1(\overline{\bs}^{(i)})=1,1\le i\le m\Bigr]^{\alpha_\gamma n}\label{eq:clauses-independent} \\
    &=\left(1-\mathbb{P}\left[\bigcup_{i\le m}(\mathcal{E}_i\cup \overline{\mathcal{E}}_i)\right]\right)^{\alpha_\gamma n}\label{eq:complement}\\
    %&\le \left(1-\sum_{1\le i\le m}\bigl(\mathbb{P}[\mathcal{E}_i]+\mathbb{P}[\overline{\mathcal{E}}_i]\bigr)+\sum_{1\le i,j\le m}\mathbb{P}\bigl[\mathcal{E}_i\cap \overline{\mathcal{E}}_j\bigr]\right)^{\alpha_\gamma n}\label{eq:bonf}\\
    &\le \left(1-2m\cdot 2^{-k}+\sum_{1\le i,j\le m}\mathbb{P}\bigl[\mathcal{E}_i\cap \overline{\mathcal{E}}_j\bigr]+\sum_{1\le i<j\le m}\mathbb{P}\bigl[\mathcal{E}_i\cap \mathcal{E}_j\bigr]+\sum_{1\le i<j\le m}\mathbb{P}\bigl[\overline{\mathcal{E}}_i\cap \overline{\mathcal{E}}_j\bigr]\right)^{\alpha_\gamma n}\label{eq:son}.
\end{align}
We now justify the lines above.~\eqref{eq:clauses-independent} follows from the fact that $\mathcal{C}_i$, $1\le i\le \alpha_\gamma n$, are independent.~\eqref{eq:complement} follows by observing that
\[
\left\{\mathcal{C}_1(\bs^{(i)}) = \mathcal{C}_1(\overline{\bs}^{(i)})=1,\forall i\in[m]\right\} = \bigcap_{i\le m}\bigl(\mathcal{E}_i^c\cap \overline{\mathcal{E}}_i^c\bigr).
\]
Finally,~\eqref{eq:son} is obtained by combining~\eqref{eq:2-to-k} with the inequality
 \begin{align*}
     & \mathbb{P}\bigl[\cup_{i\le m}(\mathcal{E}_i\cup\overline{\mathcal{E}}_i)\bigr]\\
     &\ge \sum_{1\le i\le m}\bigl(\mathbb{P}[\mathcal{E}_i]+\mathbb{P}[\overline{\mathcal{E}}_i]\bigr)-\sum_{1\le i,j\le m}\mathbb{P}\bigl[\mathcal{E}_i\cap \overline{\mathcal{E}}_j\bigr]-\sum_{1\le i<j\le m}\mathbb{P}\bigl[\mathcal{E}_i\cap \mathcal{E}_j\bigr]-\sum_{1\le i<j\le m}\mathbb{P}\bigl[\overline{\mathcal{E}}_i\cap \overline{\mathcal{E}}_j\bigr].
 \end{align*}
%where~\eqref{eq:bonf} uses Bonferroni's inequality, \eqref{eq:son} uses~\eqref{eq:2-to-k} and the symmetry $\mathbb{P}\bigl[\mathcal{E}_i\cap \mathcal{E}_j\bigr]=\mathbb{P}\bigl[\overline{\mathcal{E}}_i\cap \overline{\mathcal{E}}_j\bigr]$.
Applying Lemma~\ref{lemma:random-k-sat-prob}, we obtain that for any $1\le i<j\le m$,
\begin{equation}\label{eq:ei-ej}
\mathbb{P}\bigl[\mathcal{E}_i\cap\mathcal{E}_j\bigr] %= \mathbb{P}\bigl[\mathcal{C}_1(\sigma_i)=\mathcal{C}_1(\sigma_j)=0\bigr]=
\le 2^{-k}\cdot (1-\xi+\eta)^k.
\end{equation}
Now fix $1\le i,j\le m$. Note that if $j\notin I_i$ for $I_i$ in~\eqref{eq:index-set}, $d_H(\bs^{(i)},\overline{\bs}^{(j)})\ge (\xi-\eta) n/2$ and consequently
\[
\mathbb{P}\bigl[\mathcal{E}_i\cap \overline{\mathcal{E}}_j\bigr]\le 2^{-k}\cdot \left(1-\frac{\xi-\eta}{2}\right)^k,
\]
again by Lemma~\ref{lemma:random-k-sat-prob}. On the other hand if $j\in I_i$, then we trivially have 
\[
\mathbb{P}\bigl[\mathcal{E}_i\cap \overline{\mathcal{E}}_j\bigr]\le \mathbb{P}\bigl[\mathcal{E}_i\bigr]=2^{-k}.
\]
Recall that $|I_i|\le 1$. Thus for any $1\le i\le m$
\begin{equation*}
\sum_{1\le j\le m}\mathbb{P}\bigl[\mathcal{E}_i\cap \overline{\mathcal{E}}_j\bigr] \le 2^{-k} + (m-1)2^{-k}\left(1-\frac{\xi-\eta}{2}\right)^k,
\end{equation*}
so that
\begin{equation}\label{eq:ei-bar-ej}
    \sum_{1\le i,j\le m}\mathbb{P}\bigl[\mathcal{E}_i\cap \bar{\mathcal{E}}_j\bigr] \le m\cdot 2^{-k}+m(m-1)2^{-k}\left(1-\frac{\xi-\eta}{2}\right)^k.
\end{equation}
Combining~\eqref{eq:son} with~\eqref{eq:ei-ej} and~\eqref{eq:ei-bar-ej}, and using
\[
\left(1-\xi+\eta\right)^k\le \left(1-\frac{\xi-\eta}{2}\right)^k
\]
valid for all $0<\eta<\xi<1$, we establish Lemma~\ref{lemma:2nd-mom-pb}.
\end{proof}
Using Lemma~\ref{lemma:2nd-mom-pb}, we arrive at
\begin{align}
    A_\epsilon &\le m^2 2^n L \left(\sum_{0\le k\le \epsilon n}\binom{n}{k} \right)\cdot L \left(1-m\cdot 2^{-k}+2m(m-1)\cdot 2^{-k}\left(1-\frac{\xi-\eta}{2}\right)^k\right)^{\alpha_\gamma n} \nonumber\\
    &=m^2 2^n L^2 \binom{n}{\epsilon n}n^{O(1)} \left(1-m\cdot 2^{-k}+2m(m-1)\cdot 2^{-k}\left(1-\frac{\xi-\eta}{2}\right)^k\right)^{\alpha_\gamma n}.\label{eq:A-eps}
\end{align}
Combining~\eqref{eq:A-eps} with~\eqref{eq:lower-bd-N}, we obtain
\begin{align}
    \frac{A_\epsilon}{\mathbb{E}[N]^2} 
    &\le \frac{m^2 2^n L^2 \binom{n}{\epsilon n} n^{O(1)} \Bigl(1-m\cdot 2^{-k}+2m(m-1)\cdot 2^{-k}\left(1-\frac{\xi-\eta}{2}\right)^k\Bigr)^{\alpha_\gamma n}}{2^{2n}L^2 \Bigl(1-m\cdot 2^{-k}\Bigr)^{2\alpha_\gamma n}}\nonumber\\
    &=\exp\Bigl(-n\ln 2 + n \ln 2\cdot h_b(\epsilon)+ O(\ln n)\Bigr)\frac{\exp\Bigl(\alpha_\gamma n\ln\Bigl(1-m\cdot 2^{-k}+2m(m-1)\cdot 2^{-k}\left(1-\frac{\xi-\eta}{2}\right)^k\Bigr)\Bigr)}{\exp\Bigl(2\alpha_\gamma n\ln\Bigl(1-m\cdot 2^{-k}\Bigr)\Bigr)}\label{eq:A-eps-inter}.
\end{align}
We next employ the Taylor series, $\ln(1-x) = -x + O_x(x^2)$ valid as $x\to 0$. Note that $\xi,\eta$ and $m$ are fixed. So,  for all sufficiently large $k$, specifically when 
\[
\left(1-\frac{\xi-\eta}{2}\right)^k<\frac{1}{2(m-1)},
\]
we have
\begin{align}
    &\ln\left(1-m\cdot 2^{-k}+2m(m-1)\cdot 2^{-k}\left(1-\frac{\xi-\eta}{2}\right)^k\right)\nonumber\\
    &=-m\cdot 2^{-k} +2m(m-1)2^{-k}\left(1-\frac{\xi-\eta}{2}\right)^k +O_k\left(m^2 2^{-2k}\right)\label{tay-1}.
\end{align}
Similarly,
\begin{align}\label{tay-2}
    \ln\Bigl(1-m\cdot 2^{-k}\Bigr) = -m\cdot 2^{-k} + O_k\left(m^2 2^{-2k}\right).
\end{align}
Recall now that 
\[
\alpha_\gamma = \gamma \cdot 2^k\cdot \ln 2,\quad\text{where}\quad \gamma<\frac1m.
\]
Inserting~\eqref{tay-1} and~\eqref{tay-2} into~\eqref{eq:A-eps-inter}, we obtain
\begin{align*}
    \frac{A_\epsilon}{\mathbb{E}[N]^2} \le \exp\Bigl(n\ln 2\Bigl(-1+h_b(\epsilon)  +m\gamma +2m(m-1)(1-(\xi-\eta)/2)^k \gamma +O_k\left(m^2 2^{-k}\right)\Bigr)+O(\ln n)\Bigr).
\end{align*}
From~\eqref{eq:choice-of-eps}, we have $-\delta = -1+h_b(\epsilon)+m\gamma <0$. Now, set $K^*$ such that 
\[
\max\left\{2m(m-1)\left(1-\frac{\xi-\eta}{2}\right)^k\gamma\,,\,O_k\left(m^2 2^{-k}\right)\right\} \le -\frac{\delta}{4}.
\]
for all $k\ge K^*$. With this choice of $K^*$, we arrive at the conclusion
\begin{equation}\label{eq:A-subdominant}
\frac{A_\epsilon}{\mathbb{E}[N]^2}\le \exp\left(-\frac{\delta \ln 2}{2}n +O(\ln n)\right) = \exp\bigl(-\Theta(n)\bigr),
\end{equation}
for every $k\ge K^*$.
\subsubsection*{The Term $B_\epsilon/\mathbb{E}[N]^2$}
For $\Xi,\bar{\Xi}\in\mathcal{F}(\epsilon)$, we control the probability term by modifying Lemma~\ref{lemma:2nd-mom-pb}. 
\begin{lemma}\label{lemma:2nd-mom-pb2}
\begin{align*}
&\max_{\Xi,\bar{\Xi} \in \mathcal{F}(\epsilon)}\mathbb{P}\bigl[\Phi(\bs^{(i)})=\Phi(\overline{\bs}^{(i)})=1,\forall i\in[m]\bigr]\\
 &\le \Bigl(1-2m\cdot 2^{-k}+m(m-1)\cdot 2^{-k}\left(1-\xi+\eta\right)^k+m^2(1-\epsilon)^k 2^{-k}\Bigr)^{\alpha_\gamma n}.
\end{align*}
\end{lemma}
\begin{proof}[Proof of Lemma~\ref{lemma:2nd-mom-pb2}]
The proof is  very similar to that of Lemma~\ref{lemma:2nd-mom-pb}, we only point out the necessary modification. Note that since $\Xi,\bar{\Xi}\in\mathcal{F}(\epsilon)$, we have 
\[
d_H\bigl(\bs^{(i)},\overline{\bs}^{(j)}\bigr)>\epsilon n,\quad \forall i,j\in[m].
\]
With this, we apply Lemma~\ref{lemma:random-k-sat-prob} to conclude
\[
\mathbb{P}\bigl[\mathcal{E}_i\cap \overline{\mathcal{E}}_j\bigr] \le (1-\epsilon)^k 2^{-k},
\]
where $\mathcal{E}_i$ and $\overline{\mathcal{E}}_j$ are the events appearing in~\eqref{eq:son}. This settles Lemma~\ref{lemma:2nd-mom-pb2}.
\end{proof}
\noindent Equipped with Lemma~\ref{lemma:2nd-mom-pb2}, we now control $B_\epsilon/\mathbb{E}[N]^2$. Note that
\begin{align*}
    B_\epsilon &= \sum_{(\Xi,\bar{\Xi})\in\mathcal{F}(\epsilon)}\mathbb{P}\Bigl[\Phi(\bs^{(i)}) = \Phi(\overline{\bs}^{(i)})=1,\forall i\in[m]\Bigr] \\
    &\le 2^{2n}L^2\Bigl(1-2m\cdot 2^{-k}+m(m-1)\cdot 2^{-k}\left(1-\xi+\eta\right)^k+m^2(1-\epsilon)^k 2^{-k}\Bigr)^{\alpha_\gamma n},
\end{align*}
where we used the trivial upper bound
\[
|\mathcal{F}(\epsilon)|\le |\widehat{\mathcal{F}}(m,\xi,\eta)\times \widehat{\mathcal{F}}(m,\xi,\eta)|=2^{2n}L^2
\]
and Lemma~\ref{lemma:2nd-mom-pb2}. Using~\eqref{eq:lower-bd-N}, we arrive at
\begin{equation}\label{eq:B-eps-over-E^2}
\frac{B_\epsilon}{\mathbb{E}[N]^2}\le \frac{\Bigl(1-2m\cdot 2^{-k}+m(m-1)\cdot 2^{-k}\left(1-\xi+\eta\right)^k+m^2(1-\epsilon)^k 2^{-k}\Bigr)^{\alpha_\gamma n}}{\Bigl(1-m\cdot 2^{-k}\Bigr)^{2\alpha_\gamma n}}.
\end{equation}
\paragraph{Choice of $k$} We set $k=\Omega(\log n)$, i.e.\,$k\ge C\ln n$ for some $C>0$ to be tuned and study the each term appearing in~\eqref{eq:B-eps-over-E^2}. First, using $\alpha_\gamma = \gamma \cdot 2^k\ln 2$ and the Taylor series $\ln(1-x)=-x+O(x^2)$ as $x\to 0$,
\begin{align}
    \left(1-m\cdot 2^{-k}\right)^{2\alpha_\gamma n}&=\exp\Bigl(2\alpha_\gamma n\ln\Bigl(1-m\cdot 2^{-k}\Bigr)\Bigr)\nonumber\\
    &=\exp\Bigl(n\ln 2\cdot 2^{k+1}\cdot \gamma\Bigl(-m\cdot 2^{-k}+O(2^{-2k})\Bigr)\Bigr)\nonumber\\
    &=\exp\Bigl(n\ln 2\Bigl(-2m\gamma +O(2^{-k}) \Bigr)\Bigr).\label{eq:den}
\end{align}
Second,
\begin{align}
    &\Bigl(1-2m\cdot 2^{-k}+m(m-1)\cdot 2^{-k}\left(1-\xi+\eta\right)^k+m^2(1-\epsilon)^k 2^{-k}\Bigr)^{\alpha_\gamma n}\nonumber\\
    &=\exp\Bigl(\alpha_\gamma n\ln\Bigl(1-2m\cdot 2^{-k}+m(m-1)\cdot 2^{-k}\left(1-\xi+\eta\right)^k+m^2(1-\epsilon)^k 2^{-k}\Bigr)\Bigr)\nonumber \\
    &=\exp\Bigl(n\ln 2\cdot 2^k\cdot\gamma\Bigl(-2m\cdot 2^{-k}+m(m-1)\cdot 2^{-k}\left(1-\xi+\eta\right)^k+m^2(1-\epsilon)^k 2^{-k}+O(2^{-2k})\Bigr)\Bigr)\nonumber\\
    &=\exp\Bigl(n\ln 2\Bigl(-2m\gamma+\gamma m(m-1)\cdot 2^{-k}\left(1-\xi+\eta\right)^k+\gamma m^2(1-\epsilon)^k +O(2^{-k})\Bigr)\Bigr).\label{eq:num}
\end{align}
We combine~\eqref{eq:den} and~\eqref{eq:num} to upper bound~\eqref{eq:B-eps-over-E^2}: 
\begin{equation}\label{eq:THIS-WILL-BE-USED}
\frac{B_\epsilon}{\mathbb{E}[N]^2}\le \exp\Bigl(\gamma \ln 2 \cdot m(m-1)\cdot n(1-\xi+\eta)^k + \gamma\ln 2\cdot  m^2 \cdot n(1-\epsilon)^k + O(2^{-k}n)\Bigr).
\end{equation}
Let $k\ge C\ln n$. Note that
\[
O\left(2^{-k}n\right)=O\left(e^{-C\ln 2\ln n}e^{\ln n}\right) = O\left(e^{\ln n(1-C\ln 2)}\right)=o_n(1),
\]
provided $C>1/\ln 2$. Next, 
\[
n(1-\xi+\eta)^k = \exp\Bigl(\ln n-k\ln(1-\xi+\eta)\Bigr) \le \exp\Bigl(\ln n\Bigl(1-C\ln (1-\xi+\eta)\Bigr)\Bigr) = o_n(1)
\]
if $C>1/\ln(1-\xi+\eta)$. Finally, 
\[
n(1-\epsilon)^k =  \exp\Bigl(\ln n-k\ln(1-\epsilon)\Bigr) \le \exp\Bigl(\ln n\Bigl(1-C\ln (1-\epsilon)\Bigr)\Bigr) = o_n(1)
\]
if $C>1/\ln(1-\epsilon)$. Thus, choosing $C$ such that
\[
C>\max\left\{\frac{1}{\ln 2},\frac{1}{\ln(1-\xi+\eta)},\frac{1}{\ln(1-\epsilon)}\right\}
\]
we have that 
\begin{equation}\label{eq:B-eps-over-N-sq}
    \frac{B_\epsilon}{\mathbb{E}[N]^2}\le \exp\bigl(o_n(1)\bigr) = 1+o_n(1),
    \end{equation}
as $\gamma,m=O_n(1)$. 
\paragraph{Combining Everything}
Using~\eqref{eq:A-eps-B-eps-term}, we have
\[
\mathbb{E}[N^2]\le A_\epsilon+B_\epsilon.
\]
Using now~\eqref{eq:A-subdominant},~\eqref{eq:B-eps-over-N-sq}, and Paley-Zygmund inequality (Lemma~\ref{lemma:PZ}), we obtain
\begin{equation*}
    \mathbb{P}[N\ge 1]\ge \frac{\mathbb{E}[N]^2}{\mathbb{E}[N^2]} \ge \frac{1}{A_\epsilon/\mathbb{E}[N]^2 + B_\epsilon/\mathbb{E}[N]^2} = \frac{1}{1+o_n(1) + \exp(-\Theta(n))} = 1-o_n(1).
\end{equation*}
This concludes the proof of Theorem~\ref{thm:m-ogp-k-sat-absent}.
\subsection{Proof of Theorem~\ref{thm:m-ogp-k-sat-absent-constant-k}}\label{sec:pf-thm:m-ogp-k-sat-absent-constant-k}
Fix $\gamma<1/m$, $0<\eta<\xi<1$ and recall $Z_{\xi,\eta}$ from~\eqref{eq:Z-beta-eta}. We first establish a concentration property.
\begin{proposition}\label{prop:azuma-for-Z}
Let $M=\alpha_\gamma n$, where $\alpha_\gamma = \gamma 2^k\ln 2$. Then, for any $t_0\ge 0$, 
\begin{equation*}
\mathbb{P}\bigl[\bigl|Z_{\xi,\eta}-\mathbb{E}[Z_{\xi,\eta}]\bigr|\ge t_0\bigr]\le 2\exp\left(-\frac{t_0^2}{M}\right).
\end{equation*}
\end{proposition}
\begin{proof}[Proof of Proposition~\ref{prop:azuma-for-Z}]
Viewing $Z:= Z_{\xi,\eta}$ as a function of clauses, $Z=Z(\mathcal{C}_1,\dots,\mathcal{C}_M)$, we observe the Lipschitz property: if $\mathcal{C}_i'$ is obtained by independently resampling clause $i$, then
\[
\bigl|Z(\mathcal{C}_1,\dots,\mathcal{C}_M) - Z(\mathcal{C}_1,\dots,\mathcal{C}_{i-1},\mathcal{C}_i',\mathcal{C}_{i+1},\dots,\mathcal{C}_M)\bigr|\le 1.
\]
Proposition~\ref{prop:azuma-for-Z} then follows immediately from the bounded difference inequality~\cite[Theorem~1.1]{warnke2016method}---also known as McDiarmid's inequality.
\end{proof}
In the remainder, let $N=|\mathcal{S}_{\mathrm{k-SAT}}(\gamma,m,\xi,\eta,0)|$ and $\epsilon$ be an arbitrary value such that 
\[
\gamma m+h_b(\epsilon)-1<0.
\]
The moment calculation in the proof of Theorem~\ref{thm:m-ogp-k-sat-absent}, in particular~\eqref{eq:A-subdominant} and~\eqref{eq:THIS-WILL-BE-USED}, implies that there exists a $K^*$ such that
\begin{equation}\label{eq:2nd-mom-to-repair}  
\mathbb{P}\bigl[N\ge 1\bigr]\ge \exp\bigl(-2\alpha_\gamma nm^2 \bar{\zeta}^k\bigr),
\end{equation}
for all $k\ge K^*$, where $\bar{\zeta}$ is an arbitrary constant satisfying
\begin{equation}\label{eq:bar-Zeta}
1>\bar{\zeta}>\max\left\{\frac{1-\xi+\eta}{2},\frac{1-\epsilon}{2}\right\}
\end{equation}
For $M=n\alpha_\gamma$, $\alpha_\gamma=\gamma 2^k\ln 2$, we observe that
\[
\{N\ge 1\} = \{Z_{\xi,\eta}=M\}.
\]
Fix an arbitrary $\zeta'\in(\bar{\zeta},1)$. Notice that,
\begin{equation}\label{eq:choice-of-t-star}
    \frac{(t^*)^2}{M}\ge 2\alpha_\gamma nm^2 \bar{\zeta}^k,\quad\text{for}\quad t^* = \sqrt{3}n\alpha_\gamma m(\zeta')^{k/2}.
\end{equation}
Combining~\eqref{eq:2nd-mom-to-repair} and~\eqref{eq:choice-of-t-star} and applying Proposition~\ref{prop:azuma-for-Z} with $t_0=t^*$, we arrive at
\begin{align*}
    \mathbb{P}[Z_{\xi,\eta}=M]=\mathbb{P}[N\ge 1]&\ge \exp\bigl(-2\alpha_\gamma nm^2 \bar{\zeta}^k\bigr) \\
    &\ge 2\exp\left(-\frac{(t^*)^2}{M}\right)\ge \mathbb{P}[Z_{\xi,\eta}\ge \mathbb{E}[Z_{\xi,\eta}]+t^*],
\end{align*}
In particular, $\mathbb{E}[Z_{\xi,\eta}]\ge M-t^*$. 

We apply Proposition~\ref{prop:azuma-for-Z} once again, this time with $t_0=\sqrt{\alpha_\gamma}n$ to obtain that with probability $1-e^{-\Theta(n)}$, there is a $C>0$ such that 
\begin{align*}
    Z_{\xi,\eta}\ge \mathbb{E}[Z_{\xi,\eta}]-\sqrt{\alpha_\gamma}n &\ge n\alpha_\gamma -n\cdot\alpha_\gamma\sqrt{3}m (\zeta')^{k/2} - n\sqrt{\alpha_\gamma}\\
    &\ge n\alpha_\gamma - n\alpha_\gamma \sqrt{3}m(\zeta')^{k/2} - n2^{k/2}\sqrt{\gamma \ln 2} \\
    &\ge n\alpha_{\gamma}\Bigl(1-C\zeta^{k/2}\Bigr)
\end{align*}
for all large enough $k,n$, where $\zeta\in (\zeta',1)$ is arbitrary.\footnote{Recall that $\bar{\zeta}<\zeta'<\zeta<1$, where $\bar{\zeta}$ is defined in~\eqref{eq:bar-Zeta}, $\zeta'$ is defined in~\eqref{eq:choice-of-t-star}, and $\zeta\in(\zeta',1)$ is arbitrary.} Finally
\[
\bigl\{Z_{\xi,\eta}\ge n\alpha_\gamma(1-C\zeta^{k/2})\bigr\} = \Bigl\{\mathcal{S}_{{\mathrm{k-SAT}}}(\gamma,m,\xi,\eta,C\zeta^{k/2})\ne\varnothing\Bigr\},%\quad\text{where}\quad C =\gamma^{-\frac12}\frac{\sqrt{\ln 3}}{\ln 2}, 
\]
completing the proof of Theorem~\ref{thm:m-ogp-k-sat-absent-constant-k}.

\subsubsection*{Acknowledgments}
I am grateful to David Gamarnik for his guidance and many helpful discussions during the early stages of this work. I would like to thank Aukosh Jagannath for numerous valuable discussions regarding spin glasses, Brice Huang for carefully proofreading the draft and offering insightful comments, Lutz Warnke for a helpful discussion on the bounded differences inequality, and anonymous referees for their valuable feedback.
\bibliographystyle{amsalpha}
\bibliography{overlap/bibliom}
\end{document}